\documentclass[12pt, letter paper]{amsart}
\usepackage[utf8]{inputenc}
\usepackage{amsmath}
\usepackage{amssymb}
\usepackage{amsthm}
\usepackage[dvipsnames]{xcolor}
\usepackage{comment}

\newtheorem{theorem}{Theorem}[section]
\newtheorem{corollary}[theorem]{Corollary}
\newtheorem{lemma}[theorem]{Lemma}

\newtheorem{proposition}[theorem]{Proposition}
\newtheorem{remark}[theorem]{Remark}
\theoremstyle{definition}

\DeclareMathOperator\supp{supp}

\usepackage[luatex,colorlinks=true, linkcolor=WildStrawberry, citecolor=PineGreen]{hyperref}



\title[On the a.s. growth of $C^{\sigma}$ norms for 1d fractional BBM]{On the almost sure growth of Hölder norms for the 1d periodic fractional BBM equation}
\author{Pablo Merino}
\address[P. Merino]{Basque Center for Applied Mathematics, 48009 Bilbao, Basque Country}
\date{}

\begin{document}
	\maketitle
	
	\begin{abstract}
	We present almost sure polynomial bounds for Hölder norms of solutions of the 1d periodic fractional Benjamin-Bona-Mahony (BBM) equation. Namely, we apply quantitative quasi-invariance of certain Gaussian measures with energy cutoff using the strategy from Tzvetkov (2015) and the globalization argument from Bourgain (1994) in order to extend, almost surely, the $L^2$-based deterministic control to the $L^{\infty}$-based setting.
	\end{abstract}
	
	\tableofcontents
	
	
	
	\section{Introduction}
Let $\beta > 1$. We consider the fractional Benjamin-Bona-Mahony (BBM) equation, posed on the one-dimensional torus $\mathbb{T} := \mathbb{R}/2\pi \mathbb{Z}$:
	\begin{equation}\label{eq:cauchy-bbm}
    \left\{
	\begin{array}{l}
    \partial_t u + \partial_t |D|^{\beta} u + \partial_x u + \partial_x (u^2) = 0, \\
     u(0,x) = u_0(x),
    \end{array}
    \right.
	\end{equation}
	where $u : \mathbb{R} \times \mathbb{T} \rightarrow \mathbb{R}$ is the unknown function, $D = \frac{1}{i} \partial_x$ and
	\begin{align*}
	(|D|^{\beta}u)(x) := \sum_{n \in \mathbb{Z}} |n|^{\beta} \hat{u}(n) e^{inx}, \quad x \in \mathbb{T}.
	\end{align*}
	For each $t \in \mathbb{R}$, we denote by $\Phi(t)$ the flow from the equation in \eqref{eq:cauchy-bbm}. The mean $\int_{\mathbb{T}}udx$ is a conserved quantity of \eqref{eq:cauchy-bbm}. This allows us to restrict our discussions to mean zero $L^2$-based Sobolev spaces, which we will denote as $H^s$, given $s \geq 0$. Similarly, we will denote (restricting to mean zero elements) by $L^p$ the usual Lebesgue spaces, by $W^{\sigma,p}$ the sets of distributions on $\mathbb{T}$ which weak derivatives of order up to $\sigma$ are in $L^p$, given $\sigma \geq 0$ and $p \in [1,\infty)$, and by $C^{\sigma}$ the classical Hölder spaces, given $\sigma > 0$ and $\sigma \notin \mathbb{N}$ (see \eqref{eq:holdernorm} below). We introduce also the notation $\mathbb{Z}^* = \mathbb{Z} \setminus \{0\}$. For convenience, we define the truncated Cauchy problem as ($N \geq 1$)
	\begin{equation}\label{eq:cauchy-bbm-t}
    \left\{
	\begin{array}{l}
    \partial_t u + \partial_t |D|^{\beta} u + \partial_x u + \partial_x P_{\leq N} ((P_{\leq N}u)^2) = 0, \\
     u(0,x) = u_0(x),
    \end{array}
    \right.
	\end{equation}
	and denote its flow as $\Phi_N(t)$, where we called $P_{\leq N}$ the projection operator defined by $\widehat{P_{\leq N}v}(n) = \mathrm1_{|n| \leq N}\widehat{v}(n)$ for any $v \in L^2$. Analogously, $P_N = P_{\leq N} - P_{\leq N/2}$ and $P_{> N} = \mathrm1 - P_{\leq N}$.
	
	As it is explained in \cite{Tzvetkov2015}, this equation can be seen as a simplified model for the study of water waves, giving a correspondence to KdV-type equations for $\beta = 2$ (shallow small-amplitude long waves) and to Benjamin-Ono-type equations for $\beta = 1$ (interval waves). See \cite{L13} for more details.
	\hfill \break
	
	We are interested in studying the almost sure long-time behaviour of the norms of the solutions in \eqref{eq:cauchy-bbm} in $L^2$-based Sobolev spaces $H^{\sigma}$ and, mainly, in suitable Hölder spaces $C^{\sigma}$, in the form of polynomial in time bounds for these norms. This provides an interesting insight in the analysis of the forward cascade phenomenon \cite{Hani2017, Bourgain1996-sob, Bourgain1995-booksec} for this particular PDE, namely regarding its intersection with the study of dispersive Cauchy systems with random data, distributed according to Gaussian measures in Sobolev spaces with high regularity, which exhibit quasi-invariance under the flow \cite{Tzvetkov2015, GLT2023, GLT2023-1, Forlano2025}.
	
	The $L^2$-based Sobolev spaces are natural in order to study well-posedness of \eqref{eq:cauchy-bbm}, since the $H^{\beta/2}$-norm is formally conserved under the flow. In Lemma 2.3 of \cite{Tzvetkov2015} it was shown that, given $\beta>1$, $s \geq \frac{\beta}{2}$ and $u_0 \in H^s$, there exists a local existence time $\tau > 0$ depending only on $\|u_0\|_{H^{\frac{\beta}{2}}}$ such that a unique solution $u$ of \eqref{eq:cauchy-bbm} exists and satisfies
	\begin{align*}
	\|u(t)\|_{L^{\infty}([-\tau,\tau] , H^s)} \leq 2 \|u_0\|_{H^s}.
	\end{align*}
	Moreover, this local solution satisfies the energy conservation law (see Lemma 2.4 in \cite{Tzvetkov2015})
	\begin{align}\label{eq:con-law}
		\frac{d}{dt}\left( \|u(t)\|_{L^2}^2 + 4\pi\|u(t)\|_{H^{\frac{\beta}{2}}}^2 \right) = 0.
	\end{align}
	In particular, the $H^{\frac{\beta}{2}}$ norm of the solution remains bounded on $[-\tau,\tau]$. Therefore, we have the global-in-time bound
	\begin{align}\label{eq:det-exp-growth}
	\|u(t)\|_{H^s} \leq 2^{\lceil t/\tau \rceil + 1} \|u_0\|_{H^s}, \quad \text{ for any } t \in \mathbb{R}.
	\end{align}
	However, in order to get polynomial bounds of these Sobolev norms, more subtle estimates are needed. In Section \ref{sec:detbounds} we provide some insight on this matter based on works \cite{Tzvetkov2015} and \cite{GLT2023}. Up to the author's knowledge, the lowest known (deterministic) growth control for solutions in \eqref{eq:cauchy-bbm} were given in \cite{GLT2023} by the estimate in Proposition \ref{prop:smoothing} below, which formally reads like
	\begin{align}\label{eq:smoothing}
	\frac{d}{dt}\|u(t)\|_{H^{s + \frac{\beta}{2}}}^2 \lesssim \| u(t)\|_{H^s}^2 \|\partial_x u(t)\|_{L^{\infty}},
	\end{align}
	for $\beta > 1$ and $s > \frac{1}{2}$. By standard differentiation and interpolation, this sort of estimates allows us to bound as
	\begin{align}\label{eq:detbound}
	\|u(t)\|_{H^{s + \frac{\beta}{2}}} \lesssim \|u_0\|_{H^{s + \frac{\beta}{2}}} (1 + |t|)^{G_{\beta}(s)},
	\end{align}
	for $t \in \mathbb{R}$ and $G_{\beta}(s)$ a function of $s$ which diverges for $s \rightarrow \infty$. It is desirable to get $G_{\beta}$ independent of $s$, but neither deterministic nor probabilistic tools in this work give us such a result. 

In the deterministic setting, there is a limitation regarding the control of Hölder and $L^p$-based Sobolev norms when $p \neq 2$. Namely, recalling the Sobolev embedding $H^{\frac{1}{2}+\varepsilon} \hookrightarrow L^{\infty}$ for any $\varepsilon > 0$, from the growth control \eqref{eq:detbound} we obtain directly control of $C^{s + \frac{\beta}{2}}$-norms, as long as we assume $u_0 \in H^{s + \frac{\beta + 1}{2} + \varepsilon}$ for some $\varepsilon > 0$. In this spirit, this work provides almost sure polynomial bounds in the $L^{\infty}$ setting (in particular through Hölder norms) without such a loss in the initial datum. For that, we use quantitative quasi-invariance following the strategy from the seminal work \cite{Tzvetkov2015}. We recall that a measure $\mu$ on a locally convex space $X$ is quasi-invariant with respect to a transformation $\Phi : X \rightarrow X$ when $\mu \circ \Phi^{-1}$ is mutually absolutely continuous with respect to $\mu$.

In other words, this work deals with a particular application of the transport properties of Gaussian measures supported on high regularity Sobolev spaces when they are seen through the flow of certain Hamiltonian PDE. In the aforementioned work \cite{Tzvetkov2015}, the author provided a framework for the study of such transport properties in the form of quasi-invariance under the flow of \eqref{eq:cauchy-bbm}. It is also worth to refer to \cite{GLT2023} for the same PDE, \cite{GLT2023-1, Forlano2025} for the Benjamin-Ono-BBM case $\beta = 1$, \cite{OhTzvetkov2018, GOTW22} for the nonlinear wave equation, \cite{ForlanoTolomeo2025} for fractional nonlinear Schrödinger equations, or \cite{OhSosoeTzvetkov2018} for the cubic fourth order nonlinear Schrödinger equation, among many others. In particular, quasi-invariance becomes useful when we study almost sure dynamical properties of the PDE. For instance, we refer to almost sure scattering on $\mathbb{R}$ \cite{BurqThomann2024} or almost sure nonlinear smoothing \cite{SunTzvetkov2025} for the nonlinear Schrödinger equation, as well as almost sure global bounds for fractional nonlinear Schrödinger equations with negative regularity data \cite{ForlanoTolomeo2025}.
\hfill \break

In order to introduce the probability measure, let $\{h_n^{\omega}\}_{n > 0}$, $\{l_n^{\omega}\}_{n>0}$ be two independent sequences of independent Gaussian random variables in a probability space $(\Omega,\mathbb{P})$. Let $g_0^{\omega}$ be a standard Gaussian random variable independent of anything else and let
	\begin{align}\label{eq:def-gnomega}
	g_n^{\omega} := \begin{cases}
	\frac{1}{\sqrt{2}}(h_n^{\omega} + i l_n^{\omega}), & n \in \mathbb{N}, \\
	\frac{1}{\sqrt{2}}(h_n^{\omega} - i l_n^{\omega}), & -n \in \mathbb{N}.
	\end{cases}
    \end{align}		
Let $\beta > 1$ and $s \geq 0$, and denote by $\gamma_s$ the centered Gaussian measure on $H^s$ with covariance operator $|D|^{-s - \frac{\beta}{2}}$, i.e. the probability measure induced by the random map
\begin{align}\label{eq:inducing-map}
\omega \mapsto \varphi(\omega,x) := \sum_{n \in \mathbb{Z}^*} \frac{g_n^{\omega}}{|n|^{s+\frac{\beta}{2}}} e^{inx}.
\end{align}
Given $p \in [1,\infty)$, the spaces $W^{s + \frac{\beta - 1}{2} - \varepsilon,p}$ and $C^{s + \frac{\beta-1}{2} -\varepsilon}$ (the latter with $s + \frac{\beta-1}{2} -\varepsilon \notin \mathbb{N}$) are of $\gamma_s$-full measure for any $\varepsilon > 0$. Introducing a rigid cutoff which depends on a uniformly bounded quantity throughout the dynamics of \eqref{eq:cauchy-bbm} (recall \eqref{eq:con-law}), define the measure ($s \geq \beta/2$)
\begin{align}\label{eq:meas-rhos}
\rho_s(du) := 1_{B_{\beta/2}(R)} (u) \gamma_s(du),
\end{align}
where we denote $\mathrm1_{A}$ the characteristic function of the set $A$ and $B_{\sigma}(R)$ the centered ball of radius $R>0$ within the space $H^{\sigma}$. Moreover, we denote 
\begin{align}\label{eq:En-def}
E_N := \text{span}_{\mathbb{R}} \{(\cos(nx) , \sin(nx)) , 1 \leq |n| \leq N\},
\end{align}
and by $E_N^{\perp}$ the orthogonal complement of $E_N$ in the topology of $L^2(\mathbb{T})$. Letting $\gamma^{\perp}_{s,N}$ be the measure induced in $E_N^{\perp}$ by the map
\begin{align*}
\omega \mapsto \sum_{|n| > N} \frac{g_n^{\omega}}{|n|^{s+\frac{\beta}{2}}} e^{inx},
\end{align*}
then the measure $\gamma_s$ factorises in $E_N \times E_N^{\perp}$ as
\begin{align}\label{eq:factorization-gammas}
\gamma_s := \frac{1}{Z_N} e^{-\|P_{\leq N} u\|^2_{H^{s+\frac{\beta}{2}}}} L_N(dP_{\leq N} u) \gamma_{s,N}^{\perp} (dP_{>N} u), 
\end{align}
where $L_N$ is the Lebesgue measure induced on $E_N$ by the isomorphism between $E_N$ and $\mathbb{R}^{2N}$, and $Z_N$ is a renormalization factor. 

Going back to the Cauchy problem \eqref{eq:cauchy-bbm}, in \cite{Tzvetkov2015} it was proved that, given $\beta > \frac{4}{3}$, then for every integer $s \geq \frac{\beta}{2}$ and $t \in \mathbb{R}$ the probability measure $\gamma_s$ is quasi-invariant under the flow $\Phi(t)$. In the same spirit, given $\beta > 1$, in \cite{GLT2023} they introduced an additional exponential cutoff in order to prove quasi-invariance under $\Phi(t)$ for any $t \in \mathbb{R}$ and $s > \max\{\frac{\beta}{2} , \frac{3 - \beta}{2}\}$, providing a specific formula of the transported measures. Furthermore, considering $\beta = 1$ and using more subtle stochastic analysis, in \cite{GLT2023-1} this quasi-invariance of $\gamma_s$ was proved for any $s > 1$, and in \cite{Forlano2025} this was extended to $s > \frac{1}{2}$. 

It should be noticed that the analysis given in \cite{Tzvetkov2015, GLT2023} suffices to prove almost sure control of $H^{s+\frac{\beta-1}{2}-\varepsilon}$-norms ($\varepsilon > 0$) for data in the support of $\gamma_s$, and that the same polynomial bounds can be obtained using only deterministic tools (see Remark 7.4 in \cite{Tzvetkov2015}). Therefore, our goal is to propose an extension to the almost sure $C^{s + \frac{\beta-1}{2}-\varepsilon}$-norm control for $\varepsilon > 0$, given $\beta > 1$ and $s$ large enough.

\subsection{Statement of main results}

Regarding quasi-invariance, and namely quantitative quasi-invariance, a slight modification of the proofs in \cite{Tzvetkov2015} and \cite{GLT2023} allows us to state the following result. As they pointed out, the classical result by Ramer \cite{Ramer1974} already gives quasi-invariance for $\beta > 2$.

\begin{theorem}\label{thm:qinv}
Let $\beta > 1$, $s > \frac{\beta}{2}$ and $0 < \delta \ll 1$, and consider $R>0$ as energy cutoff in \eqref{eq:meas-rhos}. Assume also that $s > \frac{3-\beta}{2}$ if $\beta < 3$. Then $\gamma_s$ is quasi-invariant under the flow $\Phi(t)$ for any $t \in \mathbb{R}$. Furthermore, there exists some constant $C_{\delta,R}>0$ dependent on $\delta$ and $R$ such that
\begin{align}\label{eq:qinv-estimate}
\rho_s(\Phi(t)(A)) \leq \rho_s(A)^{1 - \delta} e^{C_{\delta,R}(1 + |t|)^{\frac{1}{1 - \alpha}}},
\end{align}
where
\begin{align}\label{eq:def-alpha}
\alpha := \left\{
\begin{array}{ll}
1 - \frac{3(\beta-1)}{4s}, & \text{ if } \beta < 3 \text{ and }  s > \frac{3 - \beta}{2},  \\
1 - \frac{\beta}{2s},  & \text{ if } \beta > 3.
\end{array}
\right.
\end{align}
\end{theorem}

\begin{remark}
We make the observation that the case $\beta = 3$ is not considered in Theorem \ref{thm:qinv} because none of the methods presented in its proof (see Section \ref{sec:qinv}) let us control the corresponding $L^p_{\rho_s}$ momentum, recalling that $H^{\frac{3}{2} + \varepsilon} \hookrightarrow W^{1,\infty}$ only holds for $\varepsilon > 0$, and that we cannot take the energy cutoff as in \eqref{eq:proofint-lp-3}. This would require a different scheme.
\end{remark}

Furthermore, a suitable adaptation of Bourgain's globalization argument from \cite{Bourgain1994} let us give almost sure bounds for Hölder norms of the solution in \eqref{eq:cauchy-bbm} as follows. 

\begin{theorem}\label{thm:holderbound}
Let $\beta$ and $s$ as in Theorem \ref{thm:qinv}, $\varepsilon \in (0,s - \frac{1}{2}]$ such that $s + \frac{\beta - 1}{2} - \varepsilon \notin \mathbb{N}$, and take $\delta_1 > 0$ arbitrarily small. Then for $\rho_s$-almost every $u_0$ there exists some $C>0$ such that 
\begin{align*}
\|\Phi(t)(u_0)\|_{C^{s + \frac{\beta-1}{2}-\varepsilon}} \leq C (1 + |t|)^{E(s,\beta) + \delta_1}
\end{align*}
for any $t \in \mathbb{R}$, where
\begin{align}\label{eq:def-exponent}
E(s,\beta) := \left\{
\begin{array}{ll}
\frac{2s}{3(\beta - 1)}, & \text{ if } \beta < 3 \text{ and }  s > \frac{3 - \beta}{2},  \\
\frac{s}{\beta},  &  \text{ if } \beta > 3.
\end{array}
\right.
\end{align}
\end{theorem}

\begin{remark}
Let $\beta$ and $s$ as in Theorem \ref{thm:qinv}, and $\varepsilon \in (0,s - \frac{1}{2}]$ such that $s + \frac{\beta - 1}{2} - \varepsilon \notin \mathbb{N}$. Comparing with the analogous deterministic bounds given in Corollary \ref{cor:det-growth}, in Theorem \ref{thm:holderbound} we see that for $\rho_s$-almost every datum $u_0$, which is almost surely in $H^{s + \frac{\beta - 1}{2} - \varepsilon} \setminus H^{s + \frac{\beta - 1}{2}}$, the $C^{s + \frac{\beta-1}{2} - \varepsilon}$-norm of the solution in \eqref{eq:cauchy-bbm} admits polynomial bounds which exponents are the same than the ones obtained for the $C^{s + \frac{\beta - 1}{2} - \varepsilon}$-norm of the solution with datum in $H^{s + \frac{\beta}{2}}$ when we use only deterministic tools (Corollary \ref{cor:det-growth} below and Sobolev embedding). 
\end{remark}

\begin{remark}
The upper bound $\varepsilon \leq s - \frac{1}{2}$ in Theorem \ref{thm:holderbound} comes from the fact that $\sigma \geq \frac{\beta}{2}$ is required to have a well defined global dynamics in $C^{\sigma}$ spaces for \eqref{eq:cauchy-bbm}. See Proposition \ref{prop:lwp-calpha} and Corollary \ref{cor:gwp-csigma}.
\end{remark}

\subsection{Outline of this work}

In Section \ref{sec:det-pre-calpha} we justify the existence of global dynamics of \eqref{eq:cauchy-bbm} for data in suitable Hölder spaces intersected with $B_{\beta/2}(R)$, $R > 0$. In Section \ref{sec:detbounds}, we collect some results on the deterministic control of Sobolev norms for solutions of \eqref{eq:cauchy-bbm}. From the probabilistic side, in Section \ref{sec:prob-en-est} we estimate the $L^p_{\rho_s}$ momentum of the quantity
\begin{align*}
\left|\frac{d}{dt}\| P_{\leq N} \Phi_N(t)(u_0)\|_{H^{s + \frac{\beta}{2}}}^2 \right|,
\end{align*}
which is a crucial step for the proof of quasi-invariance in the setting of \cite{Tzvetkov2015}, and namely for the computation of $\alpha$ in Theorem \ref{thm:qinv}. Once this energy estimate is obtained, in Section \ref{sec:qinv} we prove Theorem \ref{thm:qinv} using the aforementioned setting. Finally, in Section \ref{eq:globalization} we adapt Bourgain's globalization argument from \cite{Bourgain1994} to the quantitative quasi-invariance scheme from Theorem \ref{thm:qinv} in order to obtain almost sure global bounds for suitable $C^{\sigma}$-norms of the solutions of \eqref{eq:cauchy-bbm}, provided $C^{\sigma}$ lies in the support of $\gamma_s$.

\section{Global well-posedness for data in $C^{\sigma}$}\label{sec:det-pre-calpha}
	Let $\sigma > 0$, $\sigma \notin \mathbb{N}$ and $\lfloor \sigma \rfloor$ the largest nonnegative integer lower than $\sigma$. Moreover, let $\varphi \in C^{\infty}_c(\mathbb{R})$ be a radial function valued in $[0,1]$ with 
\begin{align*}
\bullet &\supp \varphi \subset \{ \xi \in \mathbb{R} : 3/4 \leq |\xi| \leq 8/3 \}, \\
\bullet &\sum_{N \in 2^{\mathbb{Z}}} \varphi(N^{-1} \xi) = 1, \text{ for any } \xi \in \mathbb{R} \setminus \{0\}, \\
\bullet &N,M \in 2^{\mathbb{Z}} \text{ such that } |\log_2(N) - \log_2(M)| \geq 2 \\
&\hspace{30mm} \Longrightarrow \supp \varphi(N^{-1} \cdot) \cap \supp \varphi(M^{-1} \cdot) = \emptyset,
\end{align*}	
and denote $\Delta_N := \varphi(N^{-1}D)$ for any $N \in 2^{\mathbb{Z}}$. We use the equivalence of norms between Besov and classical Hölder spaces (see \cite{Triebel1992, BCD2011-book}), for which the assumption $\sigma \notin \mathbb{N}$ is crucial. In this sense, consider the norm
	\begin{align}\label{eq:holdernorm}
	\|f\|_{C^{\sigma}} := \sum_{k = 0}^{\lfloor \sigma \rfloor} \| \partial_x^k f \|_{L^{\infty}} + \sup_{\substack{x,y\in\mathbb{T} \\ x \neq y}} \frac{|\partial_x^{\lfloor \sigma \rfloor}f(x) - \partial_y^{\lfloor \sigma \rfloor}f(y)|}{|x - y|^{\sigma - \lfloor \sigma \rfloor}},
\end{align}		
which is equivalent to the norm
\begin{align}\label{eq:equiv-norms}
\|f\|_{B^{\sigma}_{\infty,\infty}} := \sup_{N \in 2^{\mathbb{N}}} N^{\sigma} \| \Delta_N f\|_{L^{\infty}},
\end{align}
and which determines the Hölder space with regularity $\sigma$, denoted as $C^{\sigma}$. We stress that the reduction to mean zero elements in the periodic in space setting makes homogeneous and inhomogeneous Besov norms equivalent (namely, both of them can be written as the norm \eqref{eq:equiv-norms}). 

The main contribution of this work is to give almost sure bounds for the $C^{\sigma}$-norms of the solutions in \eqref{eq:cauchy-bbm}, for $C^{\sigma}$ within the support of the measure $\rho_s$ given in \eqref{eq:meas-rhos}. This requires to make sense of the global dynamics of \eqref{eq:cauchy-bbm} in this functional setting.
	
	Recall that the Cauchy problem \eqref{eq:cauchy-bbm} is globally well-posed with $u_0$ in the Sobolev space $H^s(\mathbb{T})$, $s \geq \frac{\beta}{2}$, with a global bound in time given by \eqref{eq:det-exp-growth} \cite{Tzvetkov2015}; and it is locally well-posed for $u_0$ in the Hölder space $C^{\sigma}$ for $\sigma \in (0,1)$ \cite{GLT2023}. Following a similar strategy, in Proposition \ref{prop:lwp-calpha} and Corollary \ref{cor:gwp-csigma} we extend the latter to global well-posedness for $\sigma \geq \frac{\beta}{2}$. We denote the free evolution of \eqref{eq:cauchy-bbm} by
	\begin{align*}
	S(t) = \exp(-t(1 + |D|^{\beta})^{-1} \partial_x ), \quad t \in \mathbb{R}.
	\end{align*}
	Before giving the main results of this Section, we provide some preliminary estimates which were already used in \cite{GLT2023}. Moreover, we notice that all the arguments about well-posedness in this section apply to negative times identically.
	
	\begin{lemma}\label{lemma:Calpha-est-calpha}
	Let $\beta > 1$ and $\varepsilon \in [0,\min\{\sigma,\beta-1\})$, and $f \in C^{\sigma}$ such that $\hat{f}(0) = 0$. Then
\begin{align}\label{eq:smoothing-calpha-beta}
\left\| \frac{\partial_x}{1 + |D|^{\beta}} f \right\|_{C^{\sigma}} \lesssim \|f\|_{C^{\sigma - \varepsilon}}.
\end{align}
In consequence, it holds that
\begin{align}\label{eq:smoothing-calpha-beta-prod}
\left\| \frac{\partial_x}{1 + |D|^{\beta}} (fg) \right\|_{C^{\sigma}} \lesssim \|f\|_{C^{\sigma - \varepsilon}}\|g\|_{L^{\infty}} + \|g\|_{C^{\sigma - \varepsilon}} \|f\|_{L^{\infty}},
\end{align}
and there exists some $C>0$ independent of $t$ such that
\begin{align}\label{eq:smoothing-calpha-beta-taylor}
\| S(t) f\|_{C^{\sigma}} \leq e^{C|t|}\|f\|_{C^{\sigma}}.
\end{align}
	\end{lemma}
	\begin{proof}
	The inequalities \eqref{eq:smoothing-calpha-beta-prod} and \eqref{eq:smoothing-calpha-beta-taylor} follow from \eqref{eq:smoothing-calpha-beta} by Leibniz fractional rule and inspection of the Taylor expansion of the exponential. Thus, it is enough to prove \eqref{eq:smoothing-calpha-beta}. Notice that, combining Lemma 2.2 from \cite{BCD2011-book} and Theorem 3.8 from Chapter 7 in \cite{SteinWeiss1971}, there exists a constant $C_{\beta} > 0$ depending only on $\beta$ such that
\begin{align*}
\left\|  \frac{\partial_x}{1 + |D|^{\beta}} \Delta_M g \right\|_{L^p} \leq C_{\beta} M^{-(\beta - 1)} \|\Delta_M g\|_{L^p},
\end{align*}	
for any $p \in [1,\infty]$, $M \in 2^{\mathbb{N}}$ and $g \in L^p$. Therefore, by \eqref{eq:equiv-norms},
\begin{align*}
\left\|  \frac{\partial_x}{1 + |D|^{\beta}} f \right\|_{C^{\sigma}} & = \sup_{N \in 2^{\mathbb{N}}} N^{\sigma} \left\|  \frac{\partial_x}{1 + |D|^{\beta}} \Delta_N f \right\|_{L^{\infty}} \\
&\leq C_{\beta} \sup_{N \in 2^{\mathbb{N}}} N^{\sigma - (\beta - 1)} \left\| \Delta_N f \right\|_{L^{\infty}}.
\end{align*}
Then, for any $\varepsilon \in [0,\min\{\sigma , \beta - 1\})$ we get that
\begin{align*}
\left\|  \frac{\partial_x}{1 + |D|^{\beta}} f \right\|_{C^{\sigma}} &\leq C_{\beta} \sup_{N \in 2^{\mathbb{N}}} N^{\sigma - \varepsilon} \left\| \Delta_N f \right\|_{L^{\infty}} = C_{\beta} \|f\|_{C^{\sigma - \varepsilon}}.
\end{align*}
\end{proof}

\begin{proposition}\label{prop:lwp-calpha}
Let $\beta > 1$, $\sigma \geq \frac{\beta}{2}$. There exists some $c > 0$ sufficiently small such that the following holds. Given $u_0 \in C^{\sigma}$, there exists a unique solution $u \in C([0,\tau],C^{\sigma})$ of \eqref{eq:cauchy-bbm-t} with initial datum $u_0$, $\tau = c/(1+\|u_0\|_{H^{\frac{\beta}{2}}})$ and
\begin{align*}
\sup_{t \in [0,\tau]} \| u(t) \|_{C^{\sigma}} \leq 2\|u_0\|_{C^{\sigma}}.
\end{align*}
The same result applies to the Cauchy problem \eqref{eq:cauchy-bbm}.
\end{proposition}

\begin{proof}
We follow the strategy of \cite{Tzvetkov2015} and Proposition 2.2 in \cite{GLT2023}. We rewrite \eqref{eq:cauchy-bbm-t} as the integral equation
\begin{align}\label{eq:int-cauchy-t}
u(t) := S(t) u_0 - \int_0^t S(t-t') \frac{\partial_x}{1 + |D|^{\beta}} P_{\leq N} ((P_{\leq N}u)^2(t')) dt'. 
\end{align} 
In the former work, it was proved that there exists a sufficiently small constant $c_{\beta} > 0$ such that, denoting $\tau_{\beta} = c_{\beta} / (1 + \|u_0\|_{H^{\frac{\beta}{2}}})$, there exists a unique solution $u$ in $C([0,\tau_{\beta}] , B_{\beta/2}(2\|u_0\|_{H^{\frac{\beta}{2}}}) )$.

In the same work, the propagation in $H^{\sigma}$ spaces was established in the sense of existence, uniqueness and continuity with respect to the initial data in $C([0,\tau'],H^{\sigma})$, with $\tau' > 0$ only depending on $\|u_0\|_{H^{\frac{\beta}{2}}}$, and
\begin{align*}
\|u\|_{L^{\infty}([0,\tau'] , H^{\sigma})} \leq 2 \|u_0\|_{H^{\sigma}}.
\end{align*}
We stress that the fact that $\tau'$ only depends on $\|u_0\|_{H^{\frac{\beta}{2}}}$, together with the conservation law \eqref{eq:con-law}, allows us to extend the solution globally in time and to give the bound in \eqref{eq:det-exp-growth}.

In the same spirit, we prove global well-posedness for data in $C^{\sigma}$. Firstly, let $u$ be the unique solution in $C([0,\tau_{\beta}] , H^{\frac{\beta}{2}})$, and take $0 < \tau \leq \tau_{\beta}$. For any $t \in [0,\tau]$ and $\varepsilon \in [0,\min\{\sigma , \beta - 1\})$, by \eqref{eq:smoothing-calpha-beta-prod} and \eqref{eq:smoothing-calpha-beta-taylor} we have that
\begin{align*}
\|u(t)\|_{C^{\sigma}} &\leq  e^{C\tau} \|u_0\|_{C^{\sigma}} + \int_0^t e^{C|t - t'|} \left\| \frac{\partial_x}{1 + |D|^{\beta}} P_{\leq N} ( (P_{\leq N}u)^2(t') ) \right\|_{C^{\sigma}} dt' \\
&\leq e^{C\tau} \|u_0\|_{C^{\sigma}} + \tau e^{C\tau} \sup_{t \in [0,\tau]} \left\| \frac{\partial_x}{1 + |D|^{\beta}} P_{\leq N} ( (P_{\leq N}u)^2 ) \right\|_{C^{\sigma}} \\
& \lesssim e^{C\tau} \|u_0\|_{C^{\sigma}} + 2\tau e^{C\tau} \sup_{t \in [0,\tau]} \|P_{\leq N}u\|_{C^{\sigma - \varepsilon}} \|P_{\leq N}u\|_{L^{\infty}} \\
&\lesssim e^{C\tau} \|u_0\|_{C^{\sigma}} + 2\tau e^{C\tau} C_1 \sup_{t \in [0,\tau]} \|P_{\leq N}u\|_{C^{\sigma}} \|P_{\leq N}u\|_{H^{\frac{\beta}{2}}},
\end{align*}
where we also used that $H^{\frac{\beta}{2}} \hookrightarrow L^{\infty}$ for any $\beta > 1$ and that
\begin{align}\label{eq:unifbddness-pn-csigma}
\|P_{\leq N}u\|_{C^{\sigma - \varepsilon}} \leq C_1 \|P_{\leq N}u\|_{C^{\sigma}}
\end{align}
with $C_1$ independent of $N$. Then, if $\tau = c_{\sigma} / (1 + \|u_0\|_{H^{\frac{\beta}{2}}})$ with $c_{\sigma} > 0$ sufficiently small depending only on $\sigma$, we get that for any $t \in [0,\tau]$
\begin{align*}
\|u(t)\|_{C^{\sigma}} \leq 2 \|u_0\|_{C^{\sigma}}.
\end{align*}
Therefore, the $C^{\sigma}$ regularity is preserved for time $\tau$. Regarding uniqueness, let $u_1$ and $u_2$ be two solutions of \eqref{eq:int-cauchy-t} in $C([0,\tau] , C^{\sigma})$. Notice that
\begin{align*}
(P_{\leq N} u_1)^2 &- (P_{\leq N} u_2)^2 \\
&= (P_{\leq N} u_1 - P_{\leq N} u_2 )(P_{\leq N} u_1 + P_{\leq N} u_2).
\end{align*}
Thus, by \eqref{eq:smoothing-calpha-beta-prod}, \eqref{eq:smoothing-calpha-beta-taylor}, $C^{\sigma} \hookrightarrow H^{\frac{\beta}{2}} \hookrightarrow L^{\infty}$ and the uniform (in $N$) boundedness from \eqref{eq:unifbddness-pn-csigma}, we obtain that for any $0 < \tau_1 \leq \tau$ and $t \in [0,\tau_1]$,
\begin{align*}
\|u_1&(t) - u_2(t)\|_{C^{\sigma}} \\
&\leq \left\| \int_0^t S(t-t') \frac{\partial_x}{1 + |D|^{\beta}} P_{\leq N} ((P_{\leq N} u_1)^2(t') - (P_{\leq N} u_2)^2(t')) dt' \right\|_{C^{\sigma}} \\
&\leq \tau_1 e^{C \tau_1} \sup_{t \in [0,\tau_1]} \left\| \frac{\partial_x}{1 + |D|^{\beta}} P_{\leq N} ((P_{\leq N} u_1)^2 - (P_{\leq N} u_2)^2 ) \right\|_{C^{\sigma}} \\
&\lesssim \tau_1 e^{C\tau_1} \sup_{t \in [0,\tau_1]} \left( \|P_{\leq N}u_1 - P_{\leq N}u_2\|_{C^{\sigma - \varepsilon}} \|P_{\leq N} u_1 + P_{\leq N} u_2\|_{L^{\infty}} \right. \\
&\hspace{30mm} \left. + \|P_{\leq N} u_1 + P_{\leq N} u_2\|_{C^{\sigma- \varepsilon}} \|P_{\leq N} u_1 - P_{\leq N} u_2\|_{L^{\infty}} \right) \\
&\leq \tau_1 e^{C\tau_1} C_1 \sup_{t \in [0,\tau_1]} \left( \|P_{\leq N}u_1 - P_{\leq N}u_2\|_{C^{\sigma}} \|P_{\leq N} u_1 + P_{\leq N} u_2\|_{H^{\frac{\beta}{2}}} \right. \\
&\hspace{30mm} \left. + \|P_{\leq N} u_1 + P_{\leq N} u_2\|_{C^{\sigma}} \|P_{\leq N} u - P_{\leq N} v\|_{H^{\frac{\beta}{2}}} \right) \\
&\leq 2 \tau_1 e^{C \tau_1} C_1 \sup_{t \in [0,\tau_1]} \|u_1 + u_2 \|_{C^{\sigma}} \cdot \sup_{t \in [0,\tau_1]} \|u_1 - u_2\|_{C^{\sigma}}.
\end{align*}
Then, if we take $\tau_1$ such that
\begin{align*}
2 \tau_1 e^{C \tau_1} C_1 (\sup_{t \in [0,\tau]} \|u_1\|_{C^{\sigma}} + \sup_{t \in [0,\tau]} \| u_2 \|_{C^{\sigma}})  < \frac{1}{2},
\end{align*}
by covering $[0,\tau]$ with intervals of size $\tau_1$ we get that $u_1 = u_2$ on $[0,\tau]$. The continuous dependence with respect to the initial data is a consequence of this analysis.
\end{proof}

Thanks to Proposition \ref{prop:lwp-calpha}, where the proper time $\tau$ only depends on $\|u_0\|_{H^{\frac{\beta}{2}}}$ for fixed $\sigma \geq \frac{\beta}{2}$, and the conservation law \eqref{eq:con-law}, we get the following result.

\begin{corollary}\label{cor:gwp-csigma}
Given $\beta > 1$ and $\sigma \geq \frac{\beta}{2}$, for every $R >0$ there exists some constant $C > 0$ such the following holds. For every $u_0 \in C^{\sigma}$ such that $\|u_0\|_{H^{\frac{\beta}{2}}} \leq R$, and every $N \geq 1$, then
\begin{align*}
\|\Phi(t)(u_0)\|_{C^{\sigma}} + \|\Phi_N(t)(u_0)\|_{C^{\sigma}} \leq e^{C(1 + |t|)}\|u_0\|_{C^{\sigma}}, \quad \text{ for any } t \in \mathbb{R}.
\end{align*}
\end{corollary}

\section{Deterministic growth of Sobolev norms}
\label{sec:detbounds}
Let $R>0$, $s \geq 0$ and $u_0 \in H^{s + \frac{\beta}{2}}$ such that $\|u_0\|_{H^{\beta/2}} \leq R$, and denote by $u_N(t) := P_{\leq N} \Phi_N(t)(u_0)$ the truncated in frequency solution of the Cauchy problem \eqref{eq:cauchy-bbm-t} with datum $u_0$. We recollect some deterministic (polynomial in $t$) global bounds for the norm in $H^{s+\frac{\beta}{2}}$ of this solution. The strategy is based on finding sub-quadratic estimates of the form
\begin{align}\label{eq:det-main-gralest}
\left|\frac{d}{dt} \|u_N(t)\|_{H^{s + \frac{\beta}{2}}}^2 \right| \lesssim R^{F(\theta)} \|u_N(t)\|_{H^{s + \frac{\beta}{2}}}^{\theta},
\end{align}
for $\theta \in (0,2)$ some constant dependent on $s$ and $\beta$, and some $F$ positive function on $\theta$. This is because \eqref{eq:det-main-gralest}, by differentiation, implies that
\begin{align*}
\|u_N(t)\|_{H^{s + \frac{\beta}{2}}} &\lesssim \left[ \|u_0\|_{H^{s + \frac{\beta}{2}}}^{2 - \theta} + \left(\frac{2 - \theta}{2}\right) R^{F(\theta)}|t| \right]^{\frac{1}{2 - \theta}} \\
&\quad \lesssim_{R,s,\beta} \|u_0\|_{H^{s + \frac{\beta}{2}}}(1 + |t|)^{\frac{1}{2 - \theta}}.
\end{align*}

We focus then on the derivation of estimates of the form \eqref{eq:det-main-gralest}. One way of obtaining them is by applying some sort of Sobolev regularity smoothing for the left-hand side of \eqref{eq:det-main-gralest}. For that, we give the smoothing estimate proved in \cite{GLT2023} which, up to our knowledge, provides the smallest known $\theta$ in \eqref{eq:det-main-gralest} for solutions of \eqref{eq:cauchy-bbm-t}. The $\frac{\beta}{2}$-smoothing in Proposition \ref{prop:smoothing} uses a crucial commutator estimate from \cite{KenigPilod2016}.

\begin{proposition}\label{prop:smoothing}
Let $\beta > 1$ and $s > \frac{1}{2}$. Then, the solution of \eqref{eq:cauchy-bbm} satisfies for all $t \in \mathbb{R}$
\begin{align*}
\left| \frac{d}{dt} \| u_N(t) \|^2_{H^{s + \frac{\beta}{2}}} \right| \lesssim \|u_N(t)\|_{H^s}^2 \| \partial_x  u_N(t)\|_{L^{\infty}}.
\end{align*}
\end{proposition}

As we said, the estimate in Proposition \ref{prop:smoothing} let us obtain an inequality of the form \eqref{eq:det-main-gralest}.

\begin{proposition}\label{prop:det-growth}
Let $\beta > 1$, $s \geq \frac{\beta}{2}$. Then, if $\beta < 3$ and $s \geq \frac{3 - \beta}{2}$,
\begin{align*}
\left|\frac{d}{dt} \|u_N(t)\|_{H^{s + \frac{\beta}{2}}}^2 \right| \lesssim R^{3 - \theta} \|u_N(t)\|_{H^{s + \frac{\beta}{2}}}^{\theta}, \quad \theta  = 2 - \frac{3 (\beta - 1)}{2s};
\end{align*}
and if $\beta \geq 3$,
\begin{align*}
&\left|\frac{d}{dt} \|u_N(t)\|_{H^{s + \frac{\beta}{2}}}^2 \right| \lesssim R^{3 - \theta_1} \|u_N(t)\|_{H^{s + \frac{\beta}{2}}}^{\theta_1}, \quad \theta_1 =
\left\{
\begin{array}{ll}
2 - \frac{\beta}{s}, &\text{ if } \beta > 3,\\
2 + \delta - \frac{3}{s},  & \text{ if } \beta = 3,
\end{array}
\right.
\end{align*}
for $\delta > 0$ an arbitrarily small constant.
\end{proposition}

\begin{proof}
We interpolate in both factors at the right-hand side of the estimate from Proposition \ref{prop:smoothing}. Regarding the norm in $H^s$, by interpolation and the conservation law \eqref{eq:con-law} we have that, for any $\beta > 1$ and $s \geq \frac{\beta}{2}$,
\begin{align}\label{eq:proofint-det-1}
\|u_N(t)\|_{H^s} \lesssim R^{\frac{\beta}{2s}} \|u_N(t)\|_{H^{s + \frac{\beta}{2}}}^{1 - \frac{\beta}{2s}}.
\end{align}
Regarding the $L^{\infty}$ norm of $\partial_x u_N(t)$, we distinguish different cases depending on the value of $\beta$.
\begin{itemize}
\item  Let $\beta \in (1,3)$. We write $\partial_xu_N(t) = \partial_x^{1 - \frac{\beta}{3}}(\partial_x^{\frac{\beta}{3}} u_N(t))$ and recall that we chose $s \geq \frac{3 - \beta}{2}$. Then we can take
\begin{align*}
\theta_3 := \frac{3 - \beta}{2s} \in \left[ \frac{1 - \frac{\beta}{3}}{s + \frac{\beta}{6}} , 1 \right],
\end{align*}
that is to say
\begin{align}\label{eq:proofint-det-3}
0 = 1 - \frac{\beta}{3} + \theta_3 \left[ \frac{1}{2} - \left(s + \frac{\beta}{6} \right) \right] + \frac{1 - \theta_3}{q}, \text{ with } q = \frac{6}{3 - \beta}.
\end{align}
As we can see in \cite{BrezisMironescu2018}, this is a valid numerology in order to apply a Gagliardo-Nirenberg inequality, from which we obtain that
\begin{align*}
\|\partial_x u_N(t) \|_{L^{\infty}} &= \|\partial_x^{1 - \frac{\beta}{3}} (\partial_x^{\frac{\beta}{3}} u_N(t)) \|_{L^{\infty}}  \nonumber \\
&\lesssim \|\partial_x^{\frac{\beta}{3}}u_N(t)\|_{H^{s + \frac{\beta}{6}}}^{\theta_3} \|\partial_x^{\frac{\beta}{3}} u_N(t)\|_{L^q}^{1-\theta_3} \nonumber \\
&= \|u_N(t)\|_{H^{s + \frac{\beta}{2}}}^{\theta_3} \|u_N(t)\|_{W^{\frac{\beta}{3} , q}}^{1 - \theta_3}.
\end{align*}
Observe that $H^{\frac{\beta}{2}} \hookrightarrow W^{\frac{\beta}{3} , q}$. Thus
\begin{align}\label{eq:proofint-det-2}
\|\partial_x u_N(t) \|_{L^{\infty}} \lesssim \|u_N(t)\|_{H^{s + \frac{\beta}{2}}}^{\theta_3} \|u_N(t)\|_{H^{\frac{\beta}{2}}}^{1 - \theta_3}.
\end{align}
\item Let $\beta > 3$. We immediately have the Sobolev embedding $H^{\frac{\beta}{2}} \hookrightarrow W^{1,\infty}$. Then
\begin{align}\label{eq:proofint-det-6}
\|\partial_x u_N(t)\|_{L^{\infty}} \lesssim R.
\end{align}
If $\beta = 3$, it should be noticed that the Sobolev embedding $H^{\frac{3}{2} + \varepsilon_1} \hookrightarrow W^{1,\infty}$ does not work for $\varepsilon_1 = 0$. For this reason, consider $\varepsilon > 0$ sufficiently small and
\begin{align}\label{eq:proofint-det-4}
\delta := \frac{\varepsilon}{\varepsilon + s - \frac{1}{2}},
\end{align}
which is in $(0,1)$ for any $s > \frac{1}{2}$. Equivalently,
\begin{align*}
0 = \delta \left[ 1 - \left(s + \frac{1}{2} \right) \right] + \varepsilon (1 - \delta).
\end{align*}
By Gagliardo-Nirenberg inequality \cite{BrezisMironescu2018}, we have that
\begin{align*}
\| \partial_x u_N(t)\|_{L^{\infty}} \lesssim \|u_N(t)\|_{W^{s+\frac{3}{2},1}}^{\delta} \|u_N(t)\|_{W^{1,\frac{1}{\varepsilon}}}^{1 - \delta}.
\end{align*}
Given that now we do have the Sobolev embedding $H^{\frac{3}{2}} \hookrightarrow W^{1,\frac{1}{\varepsilon}}$ for any $\varepsilon > 0$, and that we can take $\varepsilon$ in \eqref{eq:proofint-det-3} as small as we want, we get
\begin{align}\label{eq:proofint-det-5}
\| \partial_x u_N(t)\|_{L^{\infty}} \lesssim \|u_N(t)\|_{H^{s+\frac{3}{2}}}^{\delta} \|u_N(t)\|_{H^{\frac{3}{2}}}^{1 - \delta},
\end{align}
with $\delta > 0$ arbitrarily small.

\end{itemize}

Therefore, by \eqref{eq:proofint-det-1}, \eqref{eq:proofint-det-2}, \eqref{eq:proofint-det-6}, \eqref{eq:proofint-det-5} and the estimate from Proposition, \ref{prop:smoothing} we have proved the following. If $\beta < 3$ and $s \geq \frac{3 - \beta}{2}$,
\begin{align*}
&\left|\frac{d}{dt} \|u_N(t)\|_{H^{s + \frac{\beta}{2}}}^2 \right| \\
&\hspace{10mm} \leq \left( R^{\frac{\beta}{2s}} \|u_N(t)\|_{H^{s + \frac{\beta}{2}}}^{1 - \frac{\beta}{2s}} \right)^2  \left( \|u_N(t)\|_{H^{s + \frac{\beta}{2}}}^{\theta_3} R^{1 - \theta_3} \right);
\end{align*}
if $\beta = 3$ and $0 < \delta \ll 1$,
\begin{align*}
&\left|\frac{d}{dt} \|u_N(t)\|_{H^{s + \frac{3}{2}}}^2 \right| \\
&\hspace{10mm} \leq \left( R^{\frac{3}{2s}} \|u_N(t)\|_{H^{s + \frac{3}{2}}}^{1 - \frac{3}{2s}} \right)^2  \left( \|u_N(t)\|_{H^{s + \frac{3}{2}}}^{\delta} R^{1 - \delta} \right);
\end{align*}
and if $\beta > 3$,
\begin{align*}
&\left|\frac{d}{dt} \|u_N(t)\|_{H^{s + \frac{\beta}{2}}}^2 \right| \leq \left( R^{\frac{\beta}{2s}} \|u_N(t)\|_{H^{s + \frac{\beta}{2}}}^{1 - \frac{\beta}{2s}} \right)^2 R.
\end{align*}
Once we denote
\begin{align*}
&2 \left( 1 - \frac{\beta}{2s} \right) + \theta_3 = 2 - \frac{3(\beta - 1)}{2s} =: \theta,  \\
&2 \left( 1 - \frac{3}{2s} \right) + \delta = 2 - \frac{3}{s} + \delta =: \theta_1, \qquad \text{ for } \beta = 3, \\
&2 \left( 1 - \frac{\beta}{2s} \right) = 2 - \frac{\beta}{s} =: \theta_1, \qquad \text{ for } \beta > 3.
\end{align*}
then we can conclude.
\end{proof}

By the differentiation introduced in the beginning of this section, we have the following growth control in $H^{s + \frac{\beta}{2}}$.

\begin{corollary}\label{cor:det-growth}
Let $\beta > 1$ and $s \geq \frac{\beta}{2}$. Then, given $0 < \delta \ll 1$,
\begin{align*}
	&\|u_N(t)\|_{H^{s + \frac{\beta}{2}}} \lesssim \|u_0\|_{H^{s + \frac{\beta}{2}}} (1 + |t|)^{\frac{2s}{3(\beta - 1)}}, \quad \text{ if } \beta < 3, \\
	&\|u_N(t)\|_{H^{s + \frac{3}{2}}} \lesssim \|u_0\|_{H^{s + \frac{3}{2}}} (1 + |t|)^{\frac{s}{3 - \delta s}}, \quad \text{ if } \beta = 3, \\
	&\|u_N(t)\|_{H^{s + \frac{\beta}{2}}} \lesssim \|u_0\|_{H^{s + \frac{\beta}{2}}} (1 + |t|)^{\frac{s}{\beta}}, \quad \text{ if } \beta > 3.
	\end{align*}
\end{corollary}

\begin{remark}
In \cite{Tzvetkov2015}, given $\beta \in (4/3,2]$ and $s \geq 1$, the following energy estimate was proved:
	\begin{align}\label{eq:energy-est-nikolay}
	\left|\frac{d}{dt} \|u_N(t)\|_{H^{s + \frac{\beta}{2}}}^2 \right| \lesssim (1 + \|u_N(t)\|_{H^{\frac{\beta}{2}}}^{3 - \kappa})(1 + \||D|^{s + \frac{\beta-1}{2} - \varepsilon_2} u_N(t) \|_{L^{\infty}}^{\kappa})
	\end{align}
	for certain $\kappa \in (0,2)$ and $\varepsilon_2>0$ small. In fact, using (for any $\varepsilon_1 > 0$) the Sobolev embedding $H^{\frac{1}{2} + \varepsilon_1} \hookrightarrow L^{\infty}$ and a similar proof as the one in Proposition \ref{prop:det-growth}, one can prove that
	\begin{align*}
	\|u_N(t)\|_{H^{s+\frac{\beta}{2}}} \lesssim_{R,\kappa} (1 + |t|)^{\frac{1}{2 - \kappa}}.
	\end{align*}
	 Namely, if we further assume $s \geq 2$, then their exponent in $|t|$ takes the form
	\begin{align}\label{eq:exp-nikolay}
	\frac{1}{2 - \kappa} = \frac{2s - 1}{3 \beta - 4} + \varepsilon
	\end{align}
	for $\varepsilon > 0$ arbitrarily small. Notice that, if one compares with Corollary \ref{cor:det-growth} for $s \geq 2$ and $\beta \in (\frac{4}{3},2]$, the exponent in our work gives smaller growth.
\end{remark}

\section{Probabilistic energy estimate}
\label{sec:prob-en-est}
The argument in \cite{Tzvetkov2015, GLT2023} relies on the following differential inequality (notice the presence of the energy cutoff):

\begin{align}\label{eq:measevol-nikolay}
\frac{d}{dt} \rho_s(\Phi_N(t)(A)) \leq C p^{\alpha} (\rho_s(\Phi_N(t)(A)))^{1 - \frac{1}{p}}
\end{align}
for every $p \geq 2$, $A$ Borel set of $H^s$, $N \geq 1$, $\alpha \in (0,1)$ some constant dependent on $\beta$ and $s$, and $C > 0$ some constant independent of $N$ and $p$. We highlight two ingredients to achieve this relation: 
\begin{itemize}
\item The so-called `change of variables formula' (see the second equality in \eqref{eq:int-qinv-1} below), given in Proposition 4.1 of \cite{Tzvetkov2015} and previously in \cite{TzvetkovVisciglia2014}.

\item An $L^p_{\rho_s}$ momentum estimate for
\begin{align*}
\left.\left|\frac{d}{dt} \left\| P_{\leq N} \Phi_N(t)(u) \right\|_{H^{s + \frac{\beta}{2}}}^2 \right|_{t=0} \right|.
\end{align*}
Observe that this quantity is at the level $t = 0$, so we can exploit the Gaussian nature of the data. It is crucial that the right-hand side in the estimate of Proposition \ref{prop:smoothing} (similarly for \eqref{eq:energy-est-nikolay} in \cite{Tzvetkov2015}) involves norms which are almost surely finite in the support of $\gamma_s$.
\end{itemize}

Given $d \in \mathbb{N}$, the main probabilistic tool that we use is Gaussian concentration of sublinear and quadratic sums of independent standard Gaussian random variables $X_1, \dots, X_d$ in some probability space $(\Omega,\mathbb{P})$, namely (see \cite{Vershynin2018} for instance)
\begin{align}\label{eq:gconc-1}
\mathbb{P}\left( \left| \sum_{j=1}^d |X_j| - \mathbb{E}\left[ \sum_{j=1}^d |X_j| \right] \right| > \lambda \right) \leq c_1 \exp\left( - c_2 \frac{\lambda^2}{d}  \right)
\end{align}
and 
\begin{align}\label{eq:gconc-2}
\mathbb{P}\left( \left| \sum_{j=1}^d |X_j|^2- \mathbb{E}\left[ \sum_{j=1}^d |X_j|^2 \right] \right| > \lambda \right) \leq c_1 \exp\left( - c_2 \min\left\{ \lambda , \frac{\lambda^2}{d} \right\}  \right)
\end{align}
where $c_1,c_2 > 0$ are some suitable constants and $\lambda > 0$. They allow us to prove the following momentum bound, using the strategy from \cite{GLT2023-1}.  Given $u_0$ in the support of $\gamma_s$, denote $u_N(t) := P_{\leq N}\Phi_N(t)(u_0)$ and $u_N = u_N(0)$.

\begin{proposition}\label{prop:lpcontrol}
Let $\beta >1$, $s > \frac{\beta}{2}$ and $p \geq 1$. Then there exists some constant $C > 0$ such that
\begin{align*} 
\left\| \left.\frac{d}{dt} \left\| P_{\leq N} \Phi_N(t)(u) \right\|_{H^{s + \frac{\beta}{2}}}^2 \right|_{t=0} \right\|_{L^p_{\rho_s}} \leq C p^{\alpha},
\end{align*}
where $\alpha$ is a value in $(0,1)$ given by
\begin{align*}
\alpha := \left\{
\begin{array}{l}
1 - \frac{3(\beta-1)}{4s}, \text{ if } \beta < 3 \text{ and }  s > \frac{3 - \beta}{2},  \\
1 - \frac{\beta}{2s}, \quad \text{ if } \beta > 3.
\end{array}
\right.
\end{align*}
\end{proposition}
\begin{proof}
By Proposition \ref{prop:smoothing} and layer cake representation we have that
\begin{align}\label{eq:proofint-lp-8}
\left\| \left.\left|\frac{d}{dt} \left\| u_N \right\|_{H^{s + \frac{\beta}{2}}}^2 \right|_{t=0} \right| \right\|_{L^p_{\rho_s}}^p &\lesssim \left\| \|u_N\|_{H^s}^2 \|\partial_x u_N\|_{L^{\infty}}  \right\|_{L^p_{\rho_s}}^p \nonumber \\
&= p \int_0^{\infty} \lambda^{p-1} \rho_s(\|u_N\|_{H^s}^2 \|\partial_x u_N\|_{L^{\infty}} > \lambda) d\lambda.
\end{align}
Let $\eta \in (0,1)$. Observe that, by subadditivity of $\rho_s$,
\begin{align}\label{eq:proofint-lp-7}
\rho_s(\|u_N\|_{H^s}^2 &\|\partial_x u_N\|_{L^{\infty}} > \lambda) \nonumber \\
&\leq \rho_s(\|u_N\|_{H^s}^2 > \lambda^{\eta}) + \rho_s(\|\partial_x u_N\|_{L^{\infty}} > \lambda^{1 - \eta}).
\end{align}
Therefore, it suffices to estimate the probability of these two tail events separately. 

Let $\beta < 3$. By the $H^{\frac{\beta}{2}}$ energy cutoff in $\rho_s$ and Bernstein inequality, for any $\sigma \geq 0$, $q \in [2,\infty]$ satisfying $\frac{1}{2} - \frac{1}{q} + \sigma - \frac{\beta}{2} > 0$ and any $N_0 \in 2^{\mathbb{N}}$, we have that
\begin{align}\label{eq:probcutoff-gral}
\|P_{\leq N_0} \langle D \rangle^{\sigma} u_N\|_{L^q_x} &\lesssim \sum_{\substack{M \in 2^{\mathbb{N}} \\ M \leq \min\{N_0,N\}} } M^{\frac{1}{2} - \frac{1}{q} + \sigma - \frac{\beta}{2}}  \|P_M u_N\|_{H^{\frac{\beta}{2}}} \nonumber \\
&\leq N_0^{\frac{1}{2} - \frac{1}{q} + \sigma - \frac{\beta}{2}} \|u_N\|_{H^{\frac{\beta}{2}}} \lesssim N_0^{\frac{1}{2} - \frac{1}{q} + \sigma - \frac{\beta}{2}} R.
\end{align}

\begin{itemize}
\item Regarding the $H^s$ norm, take $q = 2$ and $\sigma = s$ in \eqref{eq:probcutoff-gral}. Thus, consider $N_1$ the largest dyadic number such that
\begin{align}\label{eq:proofint-lp-2}
N_1 \leq \left(\frac{\lambda^{\eta/2}}{\sqrt{2} R}\right)^{\frac{2}{2s - \beta}},
\end{align}
where we set $N_1 = 1$ in the case this inequality is never satisfied. Therefore, we can discard the low frequencies from the corresponding tail event in the following sense:
\begin{align}\label{eq:proofint-lp-1}
&\rho_s(\|u_N\|_{H^s}^2 > \lambda^{\eta}) \\
&\qquad  \leq \rho_s\left(\|P_{\leq N_1}u_N\|_{H^s}^2 > \frac{\lambda^{\eta}}{2} \right) + \rho_s\left(\|P_{> N_1}u_N\|_{H^s}^2 > \frac{\lambda^{\eta}}{2} \right) \nonumber \\
&\qquad = \rho_s\left(\|P_{> N_1}u_N\|_{H^s}^2 > \frac{\lambda^{\eta}}{2} \right).
\end{align}
Let $(\sigma_M)_{M \in 2^{\mathbb{N}}}$ be the sequence defined by $\sigma_M = c_0/(\log_2(M) + 1 - \log_2(N_1))^2$ for any $M \in 2^{\mathbb{N}}$ and for $c_0 > 0$ a constant small enough so that $\sum_{M \geq N_1} \sigma_M \leq 1$. Then
\begin{align*}
\rho_s\left(\|P_{> N_1}u_N\|_{H^s}^2 > \frac{\lambda^{\eta}}{2} \right) \leq \sum_{\substack{M \in 2^{\mathbb{N}} \\ N_1 \leq M \leq N} } \rho_s\left(\|P_M u_N\|_{H^s}^2 > \frac{\lambda^{\eta}\sigma_M}{2} \right).
\end{align*}
Notice that
\begin{align}\label{eq:proofint-lp-9}
\mathbb{E}_{\rho_s} \left( \|P_M u_N\|_{H^s}^2 \right) &=  \mathbb{E}_{\rho_s}\left( \sum_{|n| \sim M} \langle n \rangle^{2s} \frac{|g_n^{\omega}|^2}{|n|^{2s + \beta}} \right) \\
&\simeq M^{- \beta + 1} \left( \sum_{|n| \sim M} \frac{1}{|n|} \right) \simeq M^{- \beta + 1} \mathrm1_{M \leq N}
\end{align}
up to some constant independent of $M$. Moreover, for every $M \in 2^{\mathbb{N}}$ such that $M > N_1$, each term at the right-hand side of \eqref{eq:proofint-lp-1} is bounded by
\begin{align}\label{eq:proofint-lp-12}
&\rho_s\Big(\|P_M u_N\|_{H^s}^2 - \mathbb{E}_{\rho_s}\left( \|P_M u_N\|_{H^s}^2 \right) \nonumber  \\
&\hspace{50mm} > \frac{\lambda^{\eta}\sigma_M}{2} - M^{- \beta + 1} \mathrm1_{M \leq N} \Big) \nonumber \\
&= \rho_s\Big( \sum_{|n| \sim M} |g_n^{\omega}|^2 - \mathbb{E}_{\omega}\left( \sum_{|n| \sim M} |g_n^{\omega}|^2 \right) \nonumber  \\
&\hspace{50mm}  > M^{\beta} \left( \frac{\lambda^{\eta}\sigma_M}{2} - M^{- \beta + 1} \mathrm1_{M \leq N} \right) \Big).
\end{align}
Notice that, given $\beta > 1$, the value
\begin{align*}
C_{s,\beta} := \sup_{M \in 2^{\mathbb{N}}} \left( \frac{2 M^{ - \beta + 1} (\log_2(M) - \log_2(N_1) + 1)^2}{c_0} \right)^{\frac{1}{\eta}},
\end{align*}
is finite and only depends on $\beta$. Therefore, we can use the concentration bound \eqref{eq:gconc-2} in \eqref{eq:proofint-lp-12} in order to obtain that for any $\lambda > C_{s,\beta}$ there exists some constant $\tilde{C}_1 > 0$ independent of $M$ such that
\begin{align*}
\rho_s\left( \|P_M u_N\|_{H^s}^2 > \frac{\lambda^{\eta}\sigma_M}{2} \right) \lesssim \exp\left( - \tilde{C}_1 \frac{\lambda^{\eta}\sigma_M M^{\beta}}{2} \right).
\end{align*}
If we apply this large deviation in \eqref{eq:proofint-lp-1} and recall the choice of $N_1$ in \eqref{eq:proofint-lp-2}, we obtain that
\begin{align}\label{eq:proofint-lp-5}
\rho_s(\|u_N\|_{H^s}^2 > &\lambda^{\eta}) \lesssim \sum_{\substack{M \in 2^{\mathbb{N}} \\ N_1 \leq M \leq N} } \exp \left( - \tilde{C}_1 \frac{\lambda^{\eta}\sigma_M M^{\beta}}{2} \right) \nonumber \\
&\lesssim \exp\left( -C_1 \lambda^{\eta\left( 1 + \frac{\beta}{2s - \beta} \right)} \right) = \exp\left( - C_1 \lambda^{\frac{2s\eta}{2s - \beta}} \right)
\end{align}
for some constant $C_1 > 0$ independent of $\lambda$.

\item Regarding the $L^{\infty}$ norm of $\partial_x u_N$, let $q = \infty$ and $\sigma = 1$ in \eqref{eq:probcutoff-gral} (recall that $\beta < 3$). Analogously, consider $N_2$ the largest dyadic number such that
\begin{align}\label{eq:proofint-lp-3}
N_2 \leq \left( \frac{\lambda^{1 - \eta}}{2R} \right)^{\frac{2}{3 - \beta}},
\end{align}
where we set $N_2 = 1$ in the case this inequality is never satisfied. Thus, we can bound the second tail probability at the right-hand side of \eqref{eq:proofint-lp-7} by
\begin{align}\label{eq:proofint-lp-4}
&\rho_s\left(\|P_{\leq N_2} \partial_x u_N\|_{L^{\infty}} > \frac{\lambda^{1-\eta}}{2} \right) + \rho_s\left(\|P_{> N_2} \partial_x u_N\|_{L^{\infty}} > \frac{\lambda^{1 - \eta}}{2} \right) \nonumber \\
&\qquad  = \rho_s\left(\|P_{> N_2} \partial_x u_N\|_{L^{\infty}} > \frac{\lambda^{1 - \eta}}{2} \right).
\end{align}
Given $(\sigma_M)_{M \in 2^{\mathbb{N}}}$ as before but changing $N_1$ by $N_2$ in its definition, we obtain that
\begin{align}\label{eq:proofint-lp-30}
&\rho_s\left(\|P_{> N_2} \partial_x u_N\|_{L^{\infty}} > \frac{\lambda^{1-\eta}}{2} \right) \nonumber \\
&\qquad \leq \sum_{\substack{M \in 2^{\mathbb{N}} \\ N_2 \leq M \leq N} } \rho_s\left(\|P_M \partial_x u_N\|_{L^{\infty}} > \frac{\lambda^{1 - \eta}\sigma_M}{2} \right).
\end{align}
A difference with the $H^s$ case is that the $L^{\infty}$ norm does not allow us to use the quadratic Gaussian concentration from \eqref{eq:gconc-2}. Instead, we estimate the large deviation of the right-hand side of
\begin{align*}
\|P_M \partial_x u_N\|_{L^{\infty}} \lesssim M \sum_{|n| \sim M} \frac{|g_n^{\omega}|}{|n|^{s + \frac{\beta}{2}}},
\end{align*}
for which
\begin{align*}
\mathbb{E}_{\rho_s} &\left( M \sum_{|n| \sim M} \frac{|g_n^{\omega}|}{|n|^{s + \frac{\beta}{2}}} \right) \\
&\simeq M^{-s - \frac{\beta - 4}{2}} \mathbb{E}_{\rho_s} \left( \sum_{|n| \sim M} \frac{|g_n^{\omega}|}{|n|} \right) \simeq M^{-s - \frac{\beta - 4}{2}} \mathrm1_{M \leq N}.
\end{align*}
Therefore, for any $M \in 2^{\mathbb{N}}$, $M > N_2$, we can bound each term of the last sum in \eqref{eq:proofint-lp-30} by
\begin{align}\label{eq:proofint-lp-13}
\rho_s &\left( \|P_M  \partial_x u_N\|_{L^{\infty}} > \frac{\lambda^{1 - \eta}\sigma_M}{2} \right) \nonumber \\
&\lesssim \rho_s\Big(M \sum_{|n| \sim M} \frac{|g_n^{\omega}|}{|n|^{s + \frac{\beta}{2}}} - \mathbb{E}_{\rho_s} \left( M \sum_{|n| \sim M} \frac{|g_n^{\omega}|}{|n|^{s + \frac{\beta}{2}}} \right) \nonumber \\
& \hspace{30mm} > \frac{\lambda^{1 - \eta}\sigma_M}{2} - M^{-s - \frac{\beta - 4}{2}} \mathrm1_{M \leq N} \Big) \nonumber \\
&\lesssim \rho_s\Big( \sum_{|n| \sim M} |g_n^{\omega}| - \mathbb{E}_{\omega} \left(  \sum_{|n| \sim M} |g_n^{\omega}| \right) \nonumber \\
& \hspace{30mm} > M^{s + \frac{\beta - 2}{2}} \left( \frac{\lambda^{1 - \eta}\sigma_M}{2} - M^{-s - \frac{\beta - 4}{2}} \mathrm1_{M \leq N} \right) \Big).
\end{align}
If we assume $s > \frac{4 - \beta}{2}$, then the value
\begin{align}\label{eq:proofint-det-13}
\tilde{C}_{s,\beta} := \sup_{M \in 2^{\mathbb{N}}} \left( \frac{2 M^{\frac{4 - \beta}{2} - s} (\log_2(M) - \log_2(N_2) + 1)^2}{c_0} \right)^{\frac{1}{1-\eta}}
\end{align}
is finite and only depends on $s$ and $\beta$. Therefore, proceeding as before, for any $\lambda > \tilde{C}_{s,\beta}$ there exists some constant $\tilde{C}_2 > 0$ independent of $M$ such that, from \eqref{eq:proofint-lp-13} and the tail bound \eqref{eq:gconc-1}, we obtain
\begin{align*}
&\rho_s\left(\|P_M  \partial_x u_N\|_{L^{\infty}} > \frac{\lambda^{1 - \eta}\sigma_M}{2} \right) \nonumber \\
&\qquad \lesssim \exp\left(-\tilde{C}_2 \frac{1}{M} \left( M^{s + \frac{\beta - 2}{2}}  \frac{\lambda^{1 - \eta}\sigma_M}{2} \right)^2 \right) \nonumber \\
&\qquad = \exp\left( -\frac{\tilde{C}_2 \lambda^{2(1 - \eta)} \sigma_M^2}{4} M^{2s + \beta - 3} \right).
\end{align*}
Applying this large deviation in \eqref{eq:proofint-lp-4} and using the choice of $N_2$ in \eqref{eq:proofint-lp-3}, we have that
\begin{align}\label{eq:proofint-lp-6}
&\rho_s(\|\partial_x u_N\|_{L^{\infty}} > \lambda^{1 - \eta}) \nonumber \\
&\qquad \lesssim \sum_{\substack{M \in 2^{\mathbb{N}} \\ N_2 \leq M \leq N} }\exp\left( -\frac{\tilde{C}_2 \lambda^{2(1 - \eta)} \sigma_M^2}{4} M^{2s + \beta - 3} \right) \nonumber \\
&\qquad \lesssim \exp\left(-C_2 \lambda^{2(1 - \eta)\left( \frac{2s}{3 - \beta} \right)} \right),
\end{align}
for some constant $C_2 > 0$ independent of $\lambda$.

We consider now $s \in \left. \left( \frac{3 - \beta}{2} , \frac{4 - \beta}{2} \right] \right.$. Recall that the value $\tilde{C}_{s,\beta}$ given in \eqref{eq:proofint-det-13} is finite only when $s > \frac{4 - \beta} {2}$. Therefore, we require a new scheme. Consider without loss of generality that $\sigma_M > M^{-\varepsilon}$ for some $\varepsilon > 0$ sufficiently small satisfying $\varepsilon < s - \frac{1}{2}$, and define $N_2^*$ to be the largest $M \in 2^{\mathbb{N}}$ such that
\begin{align}\label{eq:proofint-lp-14}
M^{-s + \frac{4-\beta}{2}} \leq \frac{\sigma_M \lambda^{1 - \eta}}{2}, \quad \text{i.e.} \quad   N_2^* \simeq \lambda^{\frac{1 - \eta}{\frac{4 - \beta}{2} - s + \varepsilon}},
\end{align}
where $N_2^* = 1$ in the case this inequality is never satisfied. Notice that, from \eqref{eq:proofint-lp-3} and \eqref{eq:proofint-lp-14},
\begin{align*}
N_2 \leq N_2^* \Longleftrightarrow s > \frac{1}{2},
\end{align*}
where $\varepsilon \in (0 , s - \frac{1}{2})$ is taken as small as we need. In this regard, for $N_2 \leq M \leq N_2^*$ we can apply the procedure that worked for $s > \frac{4 - \beta}{2}$, given that for those $M$ we had $\tilde{C}_{s,\beta} < \infty$. For this reason, it is convenient to bound as
\begin{align}\label{eq:proofint-lp-15}
\rho_s &(\|\partial_x u_N\|_{L^{\infty}} > \lambda^{1 - \eta} ) \nonumber \\
&\leq \rho_s \left( \|P_{\leq N_2^*} P_{> N_2} \partial_x u_N\|_{L^{\infty}} > \frac{\lambda^{1 - \eta}}{4} \right) \nonumber \\
&\quad + \rho_s \left( \|P_{> N_2^*} \partial_x u_N\|_{L^{\infty}} > \frac{\lambda^{1 - \eta}}{4} \right),
\end{align}
so that, at the right-hand side, everything we need to control is the second term. For any $p \geq 1$ and $M \in 2^{\mathbb{N}}$ we have that
\begin{align*}
\|P_M\partial_xu\|_{L^p_{\gamma_s}} \simeq \sqrt{p} M^{\frac{3 - \beta}{2} - s},
\end{align*}
so for any $p,q \in [1,\infty)$, $p \geq q$,
\begin{align*}
\left\| \|P_M\partial_xu \right\|_{L^q_x} \|_{L^p_{\gamma_s}} \lesssim \sqrt{p} M^{\frac{3 - \beta}{2} - s}.
\end{align*}
Then, by Bernstein inequality we have that for any $\varepsilon^* > 0$
\begin{align*}
\left\| \|P_M\partial_xu \right\|_{L^\infty_x} \|_{L^p_{\gamma_s}} \lesssim \sqrt{p} M^{\frac{3 - \beta}{2} - s + \varepsilon^*}.
\end{align*}
By Markov inequality,
\begin{align*}
\gamma_s\left( \|P_M\partial_x u\|_{L^{\infty}_x} > \frac{\lambda^{1-\eta} \sigma_M}{4} \right) \leq \left( \frac{4 \sqrt{p} M^{ \frac{3 - \beta}{2} - s + \varepsilon^*}}{\lambda^{1-\eta} \sigma_M} \right)^p.
\end{align*}
Given $\lambda \geq (4e)^{\frac{1}{1 - \eta}}$, take $p$ such that
\begin{align*}
\frac{4 \sqrt{p} M^{\frac{3 - \beta}{2} - s + \varepsilon^*}}{\lambda^{1-\eta} \sigma_M} = \frac{1}{e} \Longleftrightarrow p = \frac{\sigma_M^2 \lambda^{2(1 - \eta)}}{16 e^2 M^{2\left( \frac{3 - \beta}{2} - s + \varepsilon^* \right)}},
\end{align*}
which is an available choice for $p$ because, without loss of generality, we can assume $\sigma_M \geq M^{\frac{3 - \beta}{2} + \varepsilon^* - s}$ for any $M \in 2^{\mathbb{N}}$, $M > N_2^*$ (recall that $s > \frac{3 - \beta}{2}$ and $\varepsilon^*$ can be chosen sufficiently small). Moreover, we can extend to any $\lambda > 0$ by multiplying by some constant independent of $\lambda$ and $M$ the resulting bound. Namely, 
\begin{align}\label{eq:proofint-lp-17}
&\gamma_s\left( \|P_M\partial_xu \|_{L^\infty} > \frac{\lambda^{1-\eta} \sigma_M}{4} \right) \nonumber \\
&\qquad \lesssim \exp\left(- \frac{\sigma_M^2 \lambda^{2(1 - \eta)}}{16e^2} M^{-3 - 2\varepsilon^* + \beta + 2s}\right).
\end{align}
If we plug \eqref{eq:proofint-lp-17} into the second term of the right-hand side of \eqref{eq:proofint-lp-15} and use the definition of $N_2^*$ given in \eqref{eq:proofint-lp-14}, for any $s \in \left. \left( \frac{3 - \beta}{2} , \frac{4 - \beta}{2} \right] \right.$ we get that
\begin{align*}
&\rho_s\left( \|P_{> N_2^*} \partial_x u_N\|_{L^{\infty}} > \frac{\lambda^{1 - \eta}}{4} \right) \\
&\qquad \lesssim \exp\left( - C_3 \lambda^{2(1 - \eta) \left( 1 + \frac{2s + \beta -3 - 2\varepsilon^*}{4 - \beta - 2s + 2\varepsilon} \right)} \right),
\end{align*}
for $C_3 > 0$ independent of $\lambda$. Thus, in this range of $s$ and considering also the tail bound \eqref{eq:proofint-lp-6},
\begin{align}\label{eq:proofint-lp-16}
&\rho_s(\|\partial_x u_N\|_{L^{\infty}} > \lambda^{1 - \eta})\nonumber \\
&\quad \leq \exp\left(-C_2 \lambda^{2(1 - \eta)\left( \frac{2s}{3 - \beta} \right)} \right) + \exp\left( - C_3 \lambda^{2(1 - \eta) \left( \frac{1 + 2(\varepsilon - \varepsilon^*)}{4 - \beta - 2s + 2\varepsilon} \right)} \right)\nonumber \\
&\quad \lesssim \exp\left(-C_2 \lambda^{2(1 - \eta)\left( \frac{2s}{3 - \beta} \right)} \right),
\end{align}
where we used that, given $\varepsilon^*$ and $\varepsilon$ sufficiently small, and $s > \frac{3 - \beta}{2}$, 
\begin{align*}
\frac{1}{4 - \beta - 2s + 2\varepsilon^*} \geq \frac{2s}{3 - \beta}.
\end{align*}
\end{itemize}

Thus, from \eqref{eq:proofint-lp-7}, \eqref{eq:proofint-lp-5}, \eqref{eq:proofint-lp-6} and \eqref{eq:proofint-lp-16} we have that
\begin{align*}
\rho_s(\|u_N\|_{H^s}^2 &\|\partial_x u_N\|_{L^{\infty}} > \lambda) \nonumber \\
&\lesssim  \exp\left( - C_1 \lambda^{\frac{2s\eta}{2s - \beta}} \right) +  \exp\left(-C_2 \lambda^{2(1 - \eta)\left( \frac{2s}{3 - \beta} \right)} \right).
\end{align*}
The minimum with respect to $\eta$ at the right-hand side is attained where both exponents of $\lambda$ are equal, that is to say
\begin{align*}
\frac{2s\eta}{2s - \beta} = 2(1 - \eta)\left( \frac{2s}{3 - \beta} \right) \Longleftrightarrow \eta = \frac{2(2s - \beta)}{4s - 3(\beta - 1)}.
\end{align*}
We notice that $\eta \in (0,1)$ due to $\beta < 3$ and $s > \frac{\beta}{2}$. Therefore
\begin{align*}
\rho_s(\|u_N\|_{H^s}^2 &\|\partial_x u_N\|_{L^{\infty}} > \lambda) \lesssim \exp\left( - c \lambda^{1/\alpha}  \right) \text{ where } \alpha = 1 - \frac{3(\beta - 1)}{4s},
\end{align*}
where $c > 0$ and the implicit constant are independent of $\lambda$. If we use this bound in the integral representation \eqref{eq:proofint-lp-8}, we obtain that
\begin{align}\label{eq:proofint-lp-11}
\left\| \left.\left|\frac{d}{dt} \left\| u_N \right\|_{H^{s + \frac{\beta}{2}}}^2 \right|_{t=0} \right| \right\|_{L^p_{\rho_s}}^p &\lesssim p \int_0^{\infty} \lambda^{p-1} e^{-c\lambda^{1/\alpha}} d\lambda.
\end{align}
An elementary computation shows that (recall that $\alpha p \notin \mathbb{Z}^- \cup \{0\}$)
\begin{align*}
\int_0^{\infty} \lambda^{p-1} e^{-c\lambda^{1/\alpha}} d\lambda = c^{-\alpha p} \alpha \Gamma(\alpha p),
\end{align*}
for $\Gamma(\cdot)$ the Gamma function. Moreover, by Stirling formula we have the following equivalence relation for sequences under $n \rightarrow \infty$:
\begin{align*}
\Gamma(n) \sim \sqrt{2\pi n} \left( \frac{n}{e}\right)^n.
\end{align*}
Therefore, for $p$ large enough,
\begin{align*}
\left( p \int_0^{\infty} \lambda^{p-1} e^{-c\lambda^{1/\alpha}} d\lambda \right)^{1/p} \lesssim p^{1/p} c^{-\alpha} \alpha^{1/p} \left( 2\pi p\alpha \right)^{1/2p} \left( \frac{p \alpha}{e} \right)^{\alpha} \lesssim_{\alpha} p^{\alpha}.
\end{align*}
Thus, there exists a constant $C>0$ independent of $p$ such that
\begin{align*}
\left\| \left.\left|\frac{d}{dt} \left\| u_N \right\|_{H^{s + \frac{\beta}{2}}}^2 \right|_{t=0} \right| \right\|_{L^p_{\rho_s}} \leq C p^{\alpha}
\end{align*}
This concludes the case $\beta < 3$. For $\beta > 3$, recall from \eqref{eq:proofint-det-6} that
\begin{align*}
&\|\partial_x u_N\|_{L^{\infty}} \lesssim \|u_N\|_{H^{\frac{\beta}{2}}}.
\end{align*}
By this inequality and Proposition \ref{prop:smoothing},
\begin{align*}
\frac{d}{dt} \left\| u_N \right\|_{H^{s + \frac{\beta}{2}}}^2 \lesssim R\|u_N\|_{H^s}^2.
\end{align*}
Given $\lambda > 0$, define $N_3$ as the largest dyadic number such that
\begin{align}\label{eq:proofint-lp-10}
N_3 \lesssim \left( \frac{\lambda^{1/2}}{\sqrt{2} R^\frac{3}{2}} \right)^{\frac{2}{2s - \beta}}.
\end{align}
Therefore, discarding the small frequency terms,
\begin{align}\label{eq:proofint-lp-31}
\rho_s( \|u_N\|_{H^s}^2 &\|\partial_x u_N\|_{L^{\infty}} > \lambda) \leq \sum_{\substack{ M \in 2^{\mathbb{N}} \\ N_3 < M \leq N }} \rho_s \left( \|P_M u_N\|_{H^s}^2 > \frac{\lambda}{2R}\right).
\end{align}
From \eqref{eq:proofint-lp-9} and proceeding as in the previous case (replacing $N_1$ by $N_3$ in the definition of $\sigma_M$) we bound each term of the last sum in \eqref{eq:proofint-lp-31} by
\begin{align*}
&\lesssim \rho_s\left(\|P_M u_N\|_{H^s}^2 - \mathbb{E}_{\rho_s}\left( \|P_M u_N\|_{H^s}^2 \right)  > \frac{\lambda\sigma_M}{2 R} - M^{-2s - \beta + 1} \mathrm1_{M \leq N} \right) \\
&\quad = \rho_s\left( \sum_{|n| \sim M} |g_n^{\omega}|^2 - \mathbb{E}_{\omega}\left( \sum_{|n| \sim M} |g_n^{\omega}|^2 \right)  > M^{\beta} \left( \frac{\lambda\sigma_M}{2 R} - M^{-2s - \beta + 1} \mathrm1_{M \leq N} \right) \right)\\
&\quad \lesssim \exp \left( - \tilde{C}_3 \frac{\lambda \sigma_M M^{\beta}}{2 R} \right)
\end{align*}
for some constant $\tilde{C}_3 > 0$ independent of $M$. By the choice of $N_3$ in \eqref{eq:proofint-lp-10}, we obtain that
\begin{align*}
\rho_s( \|u_N\|_{H^s}^2 &\|\partial_x u_N\|_{L^{\infty}} > \lambda) \lesssim \exp\left( -C_4 \lambda^{\frac{2s}{2s - \beta}} \right)
\end{align*}
for some constant $C_4 > 0$ independent of $\lambda$. In this way, using again \eqref{eq:proofint-lp-11} and Stirling formula, there exists some constant $C > 0$ independent of $p$ such that
\begin{align*}
\left\| \left.\left|\frac{d}{dt} \left\| u_N \right\|_{H^{s + \frac{\beta}{2}}}^2 \right|_{t=0} \right| \right\|_{L^p_{\rho_s}} \leq C p^{\alpha}
\end{align*}
where $\alpha = 1 - \frac{\beta}{2s}$. This concludes the proof.

\end{proof}

\section{Proof of Theorem \ref{thm:qinv}}
\label{sec:qinv}
In this section, we use the framework given in \cite{Tzvetkov2015} together with the $L^p_{\rho_s}$ momentum estimate from Proposition \ref{prop:lpcontrol} in order to prove Theorem \ref{thm:qinv}.

\begin{proof}[Proof of Theorem \ref{thm:qinv}]
Let $G$ be a Borel set in $H^s$. Using the aforementioned `change of variables formula' from \cite{Tzvetkov2015} and the factorisation \eqref{eq:factorization-gammas}, we have that
\begin{align}\label{eq:int-qinv-1}
\rho_s &(\Phi_N(t) (G)) = \int_{\Phi_N(t)(G)} \rho_s(du) \nonumber \\
&= \int_G \mathrm1_{B_{\beta/2}(R)}(u) \exp\left(-\frac{1}{2} \|P_{\leq N}\Phi_N(t)u\|^2_{H^{s + \frac{\beta}{2}}} \right) L_N(d P_{\leq N} u) \gamma_{s,N}^{\perp}(dP_{>N}u) \nonumber\\
&= \int_G \exp\left(\frac{1}{2} \|P_{\leq N} u\|^2_{H^{s + \frac{\beta}{2}}}-\frac{1}{2} \|P_{\leq N}\Phi_N(t)u\|^2_{H^{s + \frac{\beta}{2}}} \right) \rho_s(du).
\end{align}
Given that
\begin{align*}
t \in [0,\infty) \rightarrow \Phi_N(t)
\end{align*}
is a one parameter group of transformations, we have that
\begin{align*}
\frac{d}{dt} \left( \rho_s \circ \Phi_N(t)(A)\right)\left. \right|_{t = \overline{t}} = \frac{d}{dt} \left( \rho_s \circ \Phi_N(t)(\Phi_N(\overline{t})(A)) \right)\left. \right|_{t = 0}.
\end{align*}
Thus, using \eqref{eq:int-qinv-1} with $G = \Phi_N(\overline{t})(A)$, we obtain 
\begin{align*}
\frac{d}{dt} &\left( \rho_s (\Phi_N(t)(A))\right)\left. \right|_{t = \overline{t}} \nonumber \\
&= \left. \frac{d}{dt} \int_{\Phi_N(\overline{t})(A)} \exp\left( \frac{1}{2}\|P_{\leq N} u\|^2_{H^{s + \frac{\beta}{2}}} - \frac{1}{2}\|P_{\leq N} \Phi_N(t) u\|^2_{H^{s + \frac{\beta}{2}}} \right|_{t=0} \right) \rho_s(du) \nonumber \\
&=  \int_{\Phi_N(\overline{t})(A)}  \left( - \frac{1}{2} \left. \frac{d}{dt} \| P_{\leq N} \Phi_N(t) u\|^2_{H^{s + \frac{\beta}{2}}} \right|_{t=0} \right)\rho_s(du).
\end{align*}
Let $p \geq 1$. By Hölder's inequality, this time derivative can be bounded by
\begin{align*}
\rho_s(\Phi_N(\overline{t})(A))^{1 - \frac{1}{p}}  \left\| \left. \frac{d}{dt} \| P_{\leq N} \Phi_N(t) u\|^2_{H^{s + \frac{\beta}{2}}} \right|_{t=0} \right\|_{L^p_{\rho_s}}.
\end{align*}
Therefore, the estimate from Proposition \ref{prop:lpcontrol} implies that
\begin{align*}
\frac{d}{dt} \left( \rho_s (\Phi_N(t)(A))\right)\left. \right|_{t = \overline{t}} \leq C p^{\alpha} (\rho_s(\Phi_N(\overline{t})(A)))^{1 - \frac{1}{p}}.
\end{align*}
for $\alpha$ given in \eqref{eq:def-alpha} (recall that $\alpha \in (0,1)$ under all the considered circumstances). Thus, for any $t \in \mathbb{R}$,
\begin{align*}
\frac{d}{dt} (\rho_s(\Phi_N(t)(A))^{\frac{1}{p}}) \leq C p^{\alpha-1}.
\end{align*}
First assume that $\rho_s(A) > 0$. If we integrate in $t$
\begin{align}\label{eq:int-qinv-2}
\rho_s(\Phi_N(t)(A)) &\lesssim \left( \rho_s(A)^{\frac{1}{p}} + C p^{\alpha-1} |t|  \right)^p \nonumber \\
&=  \rho_s(A)\left( 1 + C p^{\alpha-1} |t|  \rho_s(A)^{-\frac{1}{p}}  \right)^p  \nonumber \\
&=  \rho_s(A) \exp\left(p \log ( 1 + C p^{\alpha-1} |t| \rho_s(A)^{-\frac{1}{p}} ) \right).
\end{align}
Recall that $\log(1+x) \leq x$ for any $x \geq 0$, and consider, without loss of generality, $p := p(A)$ such that
\begin{align*}
e = (2\rho_s(A))^{-\frac{1}{p}} \Longleftrightarrow p = \log\left( \frac{1}{2\rho_s(A)} \right).
\end{align*}
Then there exists some constant $C_1 > 0$ independent of $t$ and $\rho_s(A)$ for which
\begin{align*}
&\rho_s(\Phi_N(t)(A)) \lesssim  \rho_s(A) \exp\left(C_1 \left( \log\left( \frac{1}{2\rho_s(A)} \right) \right)^{\alpha} |t|  \right).
\end{align*}
Following Proposition 6.3 in \cite{GLT2023}, we use a Young inequality  in order to give a more precise evolution of $\rho_s$ under the flow $\Phi_N(t)$. Using their strategy, everything we need is that the exponent of $p$ at the right hand side of the estimate from Proposition \ref{prop:lpcontrol} is in $(0,1)$, which is our case. Namely, given $\mu > 0$ and $\tilde{\sigma} \in (0,1)$, recall the Young inequality
\begin{align*}
|a| |b| \leq \tilde{\sigma} (\mu |a|)^{\frac{1}{\sigma}} + (1 - \tilde{\sigma})\left(  \frac{|b|}{\mu} \right)^{\frac{1}{1 - \sigma}}, \quad a,b \in \mathbb{R}.
\end{align*}
Let us take
\begin{align*}
a = \left( \log\left( \frac{1}{2\rho_s(A)} \right) \right)^{\alpha}, \quad b = C_1|t|, \quad \tilde{\sigma} = \alpha, \quad 0 < \mu \ll 1.
\end{align*}
Therefore, there exists some constant $C_2 > 0$ independent of $t$ and $\rho_s(A)$ such that
\begin{align}\label{eq:int-qinv-4}
\rho_s(\Phi_N(t)(A)) &\lesssim \rho_s(A) \exp\left( \alpha \mu^{\frac{1}{\alpha}} \log \left( \frac{1}{2\rho_s(A)} \right) \right) \exp \left( \left( \frac{1 - \alpha}{\mu^{\frac{1}{1-\alpha}}}\right) (C_1 |t|)^{\frac{1}{1 - \alpha}} \right) \nonumber \\
&\leq \rho_s(A)^{1 - \alpha \mu^{\frac{1}{\alpha}}}  \exp\left( C_2 (1 +|t|)^{\frac{1}{1 - \alpha}} \right).
\end{align}
In order to extend to $N \rightarrow \infty$ at the left-hand side of \eqref{eq:int-qinv-4}, by the stability properties of the flow of \eqref{eq:cauchy-bbm} with respect to the truncated Cauchy system \eqref{eq:cauchy-bbm-t} for data in $H^{s + \frac{\beta - 1}{2}-\varepsilon} \cap B_{\beta/2}(R)$ with $\varepsilon > 0$ (Proposition 2.10 in \cite{Tzvetkov2015}), inequality \eqref{eq:int-qinv-4} and the inner regularity of the measure $\rho_s$ (for details, see Lemma 8.1 in \cite{Tzvetkov2015}), we obtain that
\begin{align*}
\rho_s(\Phi(t)(A)) \lesssim \rho_s(A)^{1 - \alpha \mu^{\frac{1}{\alpha}}}  \exp\left( C_2 (1+|t|)^{\frac{1}{1 - \alpha}} \right).
\end{align*}

On the other hand, assume $\rho_s(A) = 0$. Without loss of generality take any $\varepsilon \in (0,\frac{1}{2})$. We first prove that
\begin{align}\label{eq:int-qinv-3}
\rho_s(\Phi_N(t)(A)) < \varepsilon.
\end{align}
Due to the first inequality in \eqref{eq:int-qinv-2} and convexity (notice that $p \geq 1$), we have that
\begin{align*}
\rho_s(\Phi_N(t)(A)) \leq 2^{p-1}(\rho_s(A) + (Cp^{\alpha - 1}|t|)^p ) \leq  2^{p-1} C^p |t|^p,
\end{align*}
where we used $p^{\alpha - 1} \leq 1$. Let $p = -\log_2(\varepsilon)$ and $t_R$ such that $t_R \leq \frac{1}{4C}$. Then, for any $|t| \leq t_R$ we get \eqref{eq:int-qinv-3}. We should notice that $t_R$ only depends on $R$, $\mu$, $s$ and $\beta$, so we can extend to any $t \in \mathbb{R}$ by iterating this procedure infinitely in time and by using the fact that
\begin{align*}
t \in (\mathbb{R},+) \mapsto \Phi_N(t)
\end{align*}
is a one parameter group of transformations. To upgrade to $\Phi(t)$ from $\Phi_N(t)$ with $N < \infty$, we refer again to Lemma 8.1 in \cite{Tzvetkov2015}. 

Finally, regarding the quasi-invariance of $\gamma_s$ under $\Phi(t)$,
following \cite{Tzvetkov2015} notice that
\begin{align*}
\Phi(t)(A) = \bigcup_{R=1}^{\infty} \Phi(t)(A \cap B_{\beta/2}(R))
\end{align*}
for any Borel set $A$ of $H^s$. Assume $\gamma_s(A) = 0$, so in particular $\rho_s(A) = 0$. Thus, since 
\begin{align*}
\gamma_s(\Phi(t)(A \cap B_{\beta/2}(R))) = \rho_s(\Phi(t)(A)) = 0
\end{align*}
for any $R > 0$ and $t \in \mathbb{R}$, then by dominated convergence theorem we get the quasi-invariance of $\gamma_s$. 
\end{proof}

\begin{remark}
Let $t \in \mathbb{R}$. From Theorem \ref{thm:qinv}, we have that the measure $\rho_s \circ \Phi(t)$ is absolutely continuous w.r.t. $\rho_s$, with a density $f_s(t,u) \in L^1_{\rho_s}$. In fact, proceeding as in Proposition 6.4 from \cite{GLT2023}, one can prove that $f_s(t,u) \in L^p_{\rho_s}$ for any $p \geq 1$ and $t \in \mathbb{R}$. 
\end{remark}

\section{Proof of Theorem \ref{thm:holderbound}}
\label{eq:globalization}

The proof is an application of the quantitative quasi-invariance from Theorem \ref{thm:qinv}, large deviation estimates for $C^{s + \frac{\beta - 1}{2} - \varepsilon}$-norms ($\varepsilon > 0$), and Bourgain's globalization argument from \cite{Bourgain1994}.

\begin{lemma}\label{lemma:largedev-alpha}
Let $\beta > 1$, $s > \frac{\beta}{2}$, $\sigma \in \left(0,s + \frac{\beta-1}{2}\right)$ such that $\sigma \notin \mathbb{N}$ and $\lambda > 0$. Then there exist $C,c>0$ independent of $\lambda$ such that
\begin{align*}
\gamma_s(\|u\|_{C^{\sigma}} > \lambda) \leq C e^{-c \lambda^{2}}.
\end{align*}
\end{lemma}
\begin{proof}
Let $q \geq 1$ finite and take $Q \geq q$. Using Minkowski inequality, hypercontractivity and \eqref{eq:inducing-map}, we have that
\begin{align*}
\left\| \| u \|_{W^{\sigma,q}} \right\|_{L^Q_{\gamma_s}} \leq \left\| \| \langle D \rangle^{\sigma} u \|_{L^Q_{\gamma_s}} \right\|_{L^q_x} \lesssim Q^{\frac{1}{2}} \left( \sum_{n \in \mathbb{Z}^*} \frac{1}{\langle n \rangle^{2s + \beta - 2\sigma}} \right)^{1/2},
\end{align*}
where the last sum is finite if and only if $\sigma < s + \frac{\beta-1}{2}$. If $Q < q$, by Hölder's inequality and the same computation we obtain
\begin{align*}
\left\| \| u \|_{W^{\sigma,q}} \right\|_{L^Q_{\gamma_s}} \lesssim q^{\frac{1}{2}}.
\end{align*}
Thus, for any $Q \geq 1$,
\begin{align}\label{eq:proof-ldev1}
\left\| \| u \|_{W^{\sigma,q}} \right\|_{L^Q_{\gamma_s}} \lesssim (Qq)^{\frac{1}{2}}.
\end{align}
Therefore, by Markov inequality,
\begin{align*}
\gamma_s(\| u \|_{W^{\sigma,q}} > \lambda) \leq \left( \frac{\left\| \| u \|_{W^{\sigma,q}} \right\|_{L^Q_{\gamma_s}}}{\lambda} \right)^Q \lesssim \exp \left( Q \log\left( \frac{(Qq)^{\frac{1}{2}}}{\lambda} \right)\right).
\end{align*}
Assuming $\lambda \geq e \sqrt{q}$ take $Q$ such that $Q = \lambda^2 / (e^2 q)$, so that
\begin{align}\label{eq:proof-ldev2}
\gamma_s(\| u \|_{W^{\sigma,q}} > \lambda) \lesssim e^{-\frac{\lambda^2}{e^2q}},
\end{align}
which can be extended to $\lambda < e \sqrt{q}$ if we multiply by a suitable constant $C>0$ (independent of $\lambda$). Namely
\begin{align*}
\gamma_s(\| u \|_{W^{\sigma,q}} > \lambda)  \leq C e^{-c \lambda^2} \text{ for any } \lambda > 0.
\end{align*}

Now we consider the $C^{\sigma}$-norm. Recall that $\sigma \in \left( 0 , s + \frac{\beta-1}{2} \right)$ and $\sigma \notin \mathbb{N}$. Let $p \gg 1$ such that $\sigma + \frac{1}{p} < s + \frac{\beta - 1}{2}$, and $(\tilde{\sigma}_N)_{N \in 2^{\mathbb{N}}}$ be an element in the unit ball of $\l^1_N(2^{\mathbb{N}})$ satisfying \footnote{It would suffice to take, for each $N \in 2^{\mathbb{N}}$,
\begin{align*}
\tilde{\sigma}_N := c_{\sigma}^{-1} N^{\frac{1}{2}(\sigma + \frac{1}{p} - s - \frac{\beta - 1}{2})}, \quad c_{\sigma} = \sum_{M \in 2^{\mathbb{N}}} M^{\frac{1}{2}(\sigma + \frac{1}{p} - s - \frac{\beta - 1}{2})}.
\end{align*}}
\begin{align}\label{eq:sigmaN}
\tilde{\sigma}_N \geq N^{\sigma + \frac{1}{p} - s - \frac{\beta - 1}{2}} \text{ for any } N \geq M_0,
\end{align}
for some fixed $M_0 \in 2^{\mathbb{N}}$. By subadditivity of $\gamma_s$ we have that
\begin{align*}
\gamma_s(\| u \|_{C^{\sigma}} > \lambda) \leq \sum_{N \in 2^{\mathbb{N}}} \gamma_s(\|P_N u\|_{C^{\sigma}} > \tilde{\sigma}_N \lambda).
\end{align*}
Furthermore, by Bernstein inequality
\begin{align*}
\|P_N u\|_{C^{\sigma}} \lesssim N^{\frac{1}{p}} \|P_Nu\|_{W^{\sigma,p}}.
\end{align*}
Therefore, given $Q \geq 1$, by Markov inequality and \eqref{eq:proof-ldev1}
\begin{align*}
\gamma_s(\| u \|_{C^{\sigma}} > \lambda) &\leq \sum_{N \in 2^{\mathbb{N}}} \left( \frac{ N^{\frac{1}{p}} \left \| \| P_N u \|_{W^{\sigma,p}} \right\|_{L^Q_{\gamma_s}}}{\lambda \tilde{\sigma}_N} \right)^Q \\
&\lesssim \sum_{N \in 2^{\mathbb{N}}} \left( \frac{(Qp)^{\frac{1}{2}} N^{\sigma + \frac{1}{p} - s - \frac{\beta - 1}{2}}}{\lambda \tilde{\sigma}_N} \right)^Q \\
&= \sum_{N \in 2^{\mathbb{N}}} \exp\left( Q \log\left( \frac{(Qp)^{\frac{1}{2}} N^{\sigma + \frac{1}{p} - s - \frac{\beta - 1}{2}}}{\lambda \tilde{\sigma}_N} \right) \right).
\end{align*}
Let $\lambda \geq e \sqrt{p}$. Thus, by \eqref{eq:sigmaN},
for every $N \geq M_0$ we can always choose some $Q$ such that
\begin{align*}
\frac{(Qp)^{\frac{1}{2}} N^{\sigma + \frac{1}{p} - s - \frac{\beta - 1}{2}}}{\lambda \tilde{\sigma}_N} = \frac{1}{e} \Longleftrightarrow Q = \frac{\lambda^2 \tilde{\sigma}_N^2}{e^2 p N^{2\left( \sigma + \frac{1}{p} - s - \frac{\beta - 1}{2} \right)}}.
\end{align*}
Then
\begin{align*}
\gamma_s(\| u \|_{C^{\sigma}} > \lambda) \lesssim \sum_{N \in 2^{\mathbb{N}}: N \geq M_0} \exp \left( - \frac{\lambda^2 \tilde{\sigma}_N^2}{e^2 p N^{2\left( \sigma + \frac{1}{p} - s - \frac{\beta - 1}{2} \right)}} \right).
\end{align*}
Then, there exists some $c > 0$ such that
\begin{align*}
\gamma_s(\| u \|_{C^{\sigma}} > \lambda) \lesssim e^{-c \lambda^2},
\end{align*}
where the implicit constant does not depend on $\lambda$. To extend to $\lambda \leq e \sqrt{p}$ we argue as in the case $q < \infty$.
\end{proof}

\begin{proof}[Proof of Theorem \ref{thm:holderbound}]
As we said, we use the globalization argument from \cite{Bourgain1994}. Let $\sigma \in \left. \left[\frac{\beta}{2}, s + \frac{\beta-1}{2}\right)\right.$ such that $\sigma \notin \mathbb{N}$, $r > 0$, $t \in \mathbb{R}$, $\alpha \in (0,1)$ as it is defined in \eqref{eq:def-alpha} and $0 < \delta \ll 1$. Given $\tau > 0$ the proper time from Proposition \ref{prop:lwp-calpha}, we define the measurable sets\footnote{For negative times, we proceed identically. Thus, we reduce the globalization proof to positive times.}
\begin{align*}
A_{j,r} := \Phi_{-j\tau}(\{ u \in C^{\sigma} : \|u\|_{C^{\sigma}} \leq (1 + j\tau)^r \}), \quad j \in \mathbb{N}\cup\{0\}.
\end{align*}
By Theorem \ref{thm:qinv} and Lemma \ref{lemma:largedev-alpha} there exists some $C_{\delta,R}>0$ such that
\begin{align}\label{eq:mainbd-forcontrol}
\rho_s(A_{j,r}^c) \lesssim \exp\left( -(1 - \delta) (1 + j \tau)^{2r} + C_{\delta,R} (1 + j\tau)^{\frac{1}{1 - \alpha}} \right),
\end{align}
where the implicit constant does not depend on $j$. In order to get exponential decay, we require that $r > \frac{1}{2(1 - \alpha)}$. Indeed,
\begin{align*}
\sum_{j \geq 0} \rho_s(A_{j,r}^c) \lesssim \sum_{j \geq 0}  e^{-D_{\delta,R} (j\tau)^{2r}} < \infty,
\end{align*}
with $D_{\delta,R} > 0$ constant independent of $j$. The convergence of the last sum implies, by Borel-Cantelli Lemma, that
\begin{align*}
\rho_s(\limsup_{j \rightarrow \infty} A_{j,r}^c) = 0,
\end{align*}
i.e.
\begin{align*}
\rho_s(\liminf_{j \rightarrow \infty} A_{j,r}) = 1.
\end{align*}
Therefore, for $\rho_s$-almost every $u$ there exists some $N_u \in \mathbb{N}$ such that, for any $j \geq N_u$, $u \in A_{j,r}$. In other words, for any $j \geq N_u$ there exists some $v \in C^{\sigma}$ such that $u = \Phi(-j\tau) v$ and
\begin{align*}
    \|v\|_{C^{\sigma}} \leq (1 + j\tau)^r,
\end{align*}
so using the reversibility in time of the flow map,
\begin{align*}
\|\Phi(j\tau)u\|_{C^{\sigma}} \leq (1 + j\tau)^r.
\end{align*}
Let $\eta > 0$. Substitute $(1 + |t|)^r$ by $\eta (1 + |t|)^r$ in the previous computation. Then, for $\rho_s$-almost every $u$ there exists some $N_u \in \mathbb{N}$ such that, for $j \geq N_u$, we have that
\begin{align*}
\frac{\|\Phi(j\tau)u\|_{C^{\sigma}}}{(1 + j\tau)^r} \leq \eta,
\end{align*}
i.e. the left-hand side converges almost surely to $0$. Let $\varepsilon > 0$. By Egorov's Theorem, there exists some Borel set $D_{\varepsilon} \subset C^{\sigma}$ such that $\rho_s(D_{\varepsilon}) \geq 1 -  \varepsilon$ and that, for any $u \in D_{\varepsilon}$,
\begin{align*}
\lim_{j \rightarrow \infty} \frac{\|\Phi(j\tau)u\|_{C^{\sigma}}}{(1 + j\tau)^r} = 0.
\end{align*}
In other words, there exists some $N_{\varepsilon} \in \mathbb{N}$ (now independent of $u$) such that, for any $j \geq N_{\varepsilon}$,
\begin{align*}
\|\Phi(j\tau)u\|_{C^{\sigma}} \leq (1 + j \tau)^r,
\end{align*}
for any $u \in D_{\varepsilon}$. Thus, given $u \in D_{\varepsilon}$,
\begin{align*}
\|\Phi(j\tau)u\|_{C^{\sigma}} &\leq  \max\{(1 +  j \tau)^r , (1 + N_{\varepsilon}\tau)^r \} \\
&\lesssim (1 + j \tau)^r \max_{1 \leq j \leq N_{\varepsilon}}\{1 , N_{\varepsilon}^r / j^r \} = N_{\varepsilon}^r (1 + j \tau)^r.
\end{align*}
To extend to any $t \in \mathbb R$, we apply the local well-posedness theory from Proposition \ref{prop:lwp-calpha}. Let $t > 0$ and $j = \lfloor t / \tau \rfloor$. Then, for any $t \in (j\tau,(j+1)\tau]$ and $u \in D_{\varepsilon}$,
\begin{align*}
\|\Phi(t)u\|_{C^{\sigma}} \lesssim \|\Phi(j\tau)u\|_{C^{\sigma}},
\end{align*}
where we strongly used the fact that $\tau$ only depends on $\|u\|_{H^{\frac{\beta}{2}}}$. Therefore, 
\begin{align*}
\|\Phi(t)u\|_{C^{\sigma}} \lesssim N_{\varepsilon}^r (1 + j\tau)^r \leq N_{\varepsilon}^r (1 + |t|)^r.
\end{align*}

In order to obtain this polynomial time control $\rho_s$-almost surely, denote $(N_{1/k})_{k \in \mathbb{N}}$ a sequence defined in the construction of the Borel sets $D_{1/k}$, by picking $\varepsilon = 1/k$ for each of them, $k \in \mathbb{N}$. Define 
\begin{align*}
D := \cup_{k \in \mathbb{N}} D_{1/k}.
\end{align*}
We have that
\begin{itemize}
\item For each $u \in D$, there exists some $k_0 \in \mathbb{N}$ such that
\begin{align*}
\|\Phi(t)u\|_{C^{\sigma}} \leq N_{1/k_0}^r (1 + |t|)^r.
\end{align*}
\item Regarding its complement,
\begin{align*}
\mathbb{P}(D^c) = \mathbb{P}(\cap_{k \in \mathbb{N}} D_{1/k}^c) \leq \mathbb{P}(D_{1/k_1}^c) \leq \frac{1}{k_1}
\end{align*}
for any $k_1 \in \mathbb{N}$. Thus, $\mathbb{P}(D) = 1$. This concludes the proof.
\end{itemize}
\end{proof}

\section*{Acknowledgments}
	The author is supported by the Basque Government through the program BERC 2026-2029 (BCAM), by the project PID2024-156169NB-I00 (NCAFA), by the Severo Ochoa accreditation CEX2021-001142-S (BCAM), and by the predoctoral program of the Education Department of the Basque Government.
	
	The author thanks his PhD supervisor Renato Lucà for useful advices.



\begin{thebibliography}{10}

		
		\bibitem{BCD2011-book}
		H.~Bahouri, J.~Y.~Chemin and R.~Danchin.
		\newblock {\em Fourier analysis and nonlinear partial differential equations}.
		\newblock Grundlehren der Mathematischen Wissenschaften, Fundamental Principles of Mathematical Sciences, 343 (Springer, Heidelberg, 2011).


		\bibitem{Bogachev1998}
		V.I.~Bogachev.
		\newblock {\em Gaussian measures}.
		\newblock Mathematical surveys and monographs, 62 (American Mathematical Society, Providence, RI, 1998).


		\bibitem{Bourgain1994}
		J.~Bourgain.
		\newblock Periodic nonlinear Schrödinger equation and invariant measures.
		\newblock {\em Commun. Math. Phys.}, 166:1–--26, 1994.
		
		
		\bibitem{Bourgain1995-booksec}
		J.~Bourgain.
		\newblock Aspects of Long Time Behaviour of Solutions of Nonlinear Hamiltonian Evolution Equations.
		\newblock {\em Eliashberg Y, Milman V, Polterovich L, Schoen R, eds. Geometries in Interaction: GAFA Special Issue in Honor of Mikhail Gromov. Birkhäuser}, 1995:105–--140, 1995.	
		
		\bibitem{Bourgain1996-sob}
		J.~Bourgain.
		\newblock On the growth in time of higher Sobolev norms of smooth solutions of Hamiltonian PDE.
		\newblock {\em Int Math Res Notices}, 1996(6):277–--304, 1996.	
		
		
		
		\bibitem{BurqThomann2024}
		N.~Burq and L.~Thomann.
		\newblock Almost Sure Scattering for the One Dimensional Nonlinear Schrödinger Equation.
		\newblock {\em Memoirs of the AMS}, 296(1480), 2024.		
		
		
		
		
		\bibitem{BrezisMironescu2018}
		H.~Brezis and P.~Mironescu.
		\newblock Gagliardo–Nirenberg inequalities and non-inequalities: The full story.
		\newblock {\em Ann. Inst. Henri Poincare (C) Anal. Non Lineaire}, 35(5):1355---1376, 2018.	

		\bibitem{Forlano2025}
		J.~Forlano.
		\newblock Improved quasi-invariance result for the periodic Benjamin-Ono-BBM equation.
		\newblock {\em  	arXiv:2501.17180}, 2025.		
		
		
		\bibitem{ForlanoTolomeo2025}
		J.~Forlano and L.~Tolomeo.
		\newblock Quasi-invariance of Gaussian measures of negative regularity for fractional nonlinear Schrödinger equations.
		\newblock {\em J. Eur. Math. Soc.}, 2025.		
		
		\bibitem{GOTW22}
		T.~Gunaratnam, T.~Oh, N.~Tzvetkov and H.~Weber.
		\newblock Quasi-invariant Gaussian measures for the nonlinear wave equation in three dimensions.
		\newblock {\em Probab. Math. Phys.}, 3(2):343–--379, 2022.
	
		\bibitem{GLT2023}
		G. ~Genovese, R. ~Lucà, N. ~Tzvetkov. 
		\newblock Transport of Gaussian measures with exponential cut-off for Hamiltonian PDEs. 
		\newblock {\em J. Anal. Math.}, 150(2), 737---787, 2023.
		
		\bibitem{GLT2023-1}
		G. ~Genovese, R. ~Lucà, N. ~Tzvetkov. 
		\newblock Quasi-invariance of Gaussian measures for the periodic Benjamin-Ono-BBM equation. 
		\newblock {\em Stoch. Partial Differ. Equ. Anal. Comput.}, 11(2), 651---684, 2023.
		
		\bibitem{Hani2017}
		Z. ~Hani. 
		\newblock Out-of-equilibrium dynamics and statistics of dispersive PDE. 
		\newblock {\em JEDP}, 1---12, 2017.
				
		\bibitem{KenigPilod2016}
		C.E. ~Kenig, D. ~Pilod. 
		\newblock Local well-posedness for the KdV hierarchy at high regularity. 
		\newblock {\em Adv. Differential Equations}, 21 (9/10), 801---836, 2016.		
		
		
		
		
		\bibitem{L13}
		D.~Lannes.
		\newblock {\em The Water Waves Problem: Mathematical Analysis and Asymptotics}.
		\newblock Mathematical Surveys and Monographs, 188 (American Mathematical Society, 2013).
		
		\bibitem{OhTzvetkov2018}
		T. ~Oh, N. ~Tzvetkov. 
		\newblock Quasi-invariant Gaussian measures for the two-dimensional defocusing cubic nonlinear wave equation. 
		\newblock {\em J. Eur. Math. Soc.}, 22(6), 1785---1826, 2020.
		
		
		
		
		\bibitem{OhSosoeTzvetkov2018}
		T. ~Oh, P. ~Sosoe, N. ~Tzvetkov. 
		\newblock An optimal regularity result on the quasi-invariant Gaussian measures for the cubic fourth order nonlinear Schrödinger equation. 
		\newblock {\em J. Éc. polytech., Math.}, 5, 793---841, 2018.	
		
		
		\bibitem{Ramer1974}
		R. ~Ramer. 
		\newblock On nonlinear transformations of Gaussian measures. 
		\newblock {\em J. Funct. Anal.}, 15(2), 166---187, 1974.			
		

		\bibitem{Vershynin2018}
		R.~Vershynin.
		\newblock {\em High-Dimensional Probability: An Introduction with Applications in Data Science}.
		\newblock Cambridge Series in Statistical and Probabilistic Mathematics (Cambridge: Cambridge University Press, 2018).
		
		
	
		
		
		\bibitem{SunTzvetkov2025}
		C.~Sun, N.~Tzvetkov. 
		\newblock Almost sure global nonlinear smoothing for the 2D NLS.
		\newblock {\em arxiv.org/pdf/2512.13088}, 2025.	
		
		\bibitem{SteinWeiss1971}
		E.M.~Stein and G.~Weiss.
		\newblock {\em Introduction to Fourier Analysis on Euclidean Spaces}.
		\newblock Princeton Mathematical Series. Princeton University Press, 1971.
		
		\bibitem{Triebel1992}
		H.~Triebel.
		\newblock {\em Theory of Function Spaces II}.
		\newblock Princeton Mathematical Series. Springer, 1992.
		
		
		
		\bibitem{Tzvetkov2015}
		N. ~Tzvetkov. 
		\newblock Quasi-invariant Gaussian measures for one-dimensional Hamiltonian partial differential equations. 
		\newblock {\em Forum Math. Sigma}, 3, 2015.	
	
		
		\bibitem{TzvetkovVisciglia2014}
		N. ~Tzvetkov, N. ~Visciglia. 
		\newblock Invariant Measures and Long-Time Behavior for the Benjamin–Ono Equation. 
		\newblock {\em Int. Math. Res. Not.}, 2014(17), 4679---4714, 2014.
		
	\end{thebibliography}
\end{document}